\documentclass[a4paper,12pt]{amsart}

\usepackage{amsfonts}
\usepackage{amsmath}
\usepackage{amssymb}
\usepackage{mathrsfs}
\usepackage{hyperref}
\usepackage{graphicx}

\usepackage[usenames]{color}
\newtheorem{thm}{Theorem}[section]

\newtheorem{lem}[thm]{Lemma}

\numberwithin{equation}{section}
\theoremstyle{definition}

\begin{document}
%%%%%%%%%%%%%%%%%%%%%%%%%%%%%%%%%%%%%%%%%%%%%
%%%%%%%%%%%%%%%%%%%%%%%%%%%%%%%%%%%%%%%%%%%%%

\title[Some estimates for Mittag-Leffler function in quantum calculus]{ Some estimates for Mittag-Leffler function in quantum calculus and applications}

\author[Michael Ruzhansky]{Michael Ruzhansky}
\address{
 Michael Ruzhansky:
  \endgraf
Department of Mathematics: Analysis, Logic and Discrete Mathematics,
  \endgraf
 Ghent University, Ghent,
 \endgraf
  Belgium 
  \endgraf
  and 
 \endgraf
 School of Mathematical Sciences, Queen Mary University of London, London,
 \endgraf
 UK  
 \endgraf
  {\it E-mail address} {\rm michael.ruzhansky@ugent.be}
  }

\author[Serikbol Shaimardan]{Serikbol Shaimardan}
\address{
  Serikbol Shaimardan:
  \endgraf
  L. N. Gumilyov Eurasian National University, Astana,
  \endgraf
  Kazakhstan 
  \endgraf
  and 
  \endgraf
Department of Mathematics: Analysis, Logic and Discrete Mathematics
  \endgraf
 Ghent University, Ghent,
 \endgraf
  Belgium
  \endgraf
  {\it E-mail address} {\rm shaimardan.serik@gmail.com} 
  }

\author[N. Tokmagambetov ]{Niyaz Tokmagambetov }
\address{
  Niyaz Tokmagambetov:
  \endgraf 
  Centre de Recerca Matem\'atica
  \endgraf
  Edifici C, Campus Bellaterra, 08193 Bellaterra (Barcelona), Spain
  \endgraf
  and
  \endgraf   
  Institute of Mathematics and Mathematical Modeling
  \endgraf
  125 Pushkin str., 050010 Almaty, Kazakhstan
  \endgraf  
  {\it E-mail address:} {\rm tokmagambetov@crm.cat; tokmagambetov@math.kz}
  }

\thanks{The authors are supported by the FWO Odysseus 1 grant G.0H94.18N: Analysis and Partial Differential Equations and by the Methusalem programme of the Ghent University Special Research Fund (BOF) (Grant number 01M01021). Michael Ruzhansky is also supported by EPSRC grant EP/R003025/2, and  the second author  by the international internship program “Bolashak” of the Republic of Kazakhstan.}
\date{}

\begin{abstract}
The study of the Mittag-Leffler function and its various generalizations has become a very popular topic in mathematics and its applications.  In the present paper we prove the following estimate for the $q$-Mittag-Leffler function: 
\begin{eqnarray*} 
\frac{1}{1+\Gamma_q\left(1-\alpha\right)z}\leq e_{\alpha,1}\left(-z;q\right)\leq\frac{1}{1+\Gamma_q\left(\alpha+1\right)^{-1}z}.
\end{eqnarray*}
for all $0 < \alpha < 1$ and $z>0$.

Moreover, we apply it to investigate the solvability results for direct and inverse problems for time-fractional pseudo-parabolic equations in  quantum calculus for a large class of positive operators with discrete spectrum.  
\end{abstract}



\maketitle

{\section{Introduction}}
%%%%%%%%%%%%%%%%%%%%%%%%%%%%%%%%%%%%%%%%%%%%%%%%%%%%%%%%%%%%%%%
%%%%%%%%%%%%%%%%%%%%%%%%%%%%%%%%%%%%%%%%%%%%%%%%%%%%%%%%%%%%%%%

The theory of  quantum calculus analyses the basic structures of physics from the point of view of non-commutative geometry on  quantum groups. We can obtain a quantum group by considering the classical algebra of observables and quantising it relative to some Poisson bracket.  It is well-known that quantum algebras and quantum groups support the  theory  of several physical phenomena. They  play an important role in the conformal field theory,  in statistical physics \cite{Baxte1982}   and in a wide range of applications, from cosmic strings and black holes to solid state physics problems   \cite{Anyons1992},   \cite{BD1999}, \cite{LSN2006} and \cite{Lavagno} and \cite{Wilczek1990}.

In \cite{Biedenharn} and \cite{Macfarlane},   L. Biedenharn and A. Macfarlane  showed  that the $q$-calculus  plays a central role in the representation of the quantum groups    with a deep physical meaning and is not merely a mathematical exercise. In the framework of the $q$-Heisenberg algebra, $q$-deformed Schr\"odinger equations have
been proposed in \cite{Zhang1} and \cite{Zhang2}.  A number of articles address the   further developments and recent results in the  $q$-deformed calculus
 (see e.g. \cite{AM2012}, \cite{FBB2006}, \cite{Ernst2002},  \cite{MOP2014},   \cite{shaimardan2015}, \cite{shaimardan2016} and we refer to the books \cite{Ernst2012} and  \cite{GR1990}).

 Recently,  it is also  well recognized that the Mittag-Leffler function plays a fundamental role in the fractional calculus problems.  In 1903, the Swedish mathematician G\"osta Mittag-Leffler discovered \cite{Podlubny1999}   the following function   defined by
\begin{eqnarray*} 
E_{\alpha}\left(z\right)=
\sum\limits_{k=0}^{\infty}\frac{z^{k}}{\Gamma(\alpha k+1)}, \;\;\; \Re(\alpha)>0,z\in\mathbb{C},
\end{eqnarray*}
which is convergent in the
whole complex plane. In \cite{W1905} (see   \cite[Subsection 1.8]{KST2006}), A.~Wiman firstly introduced the generalization of the Mittag-Leffler as a two-parameter version $E_{\alpha,\beta}\left(z\right)$   defined by  
\begin{eqnarray*} 
E_{\alpha,\beta}\left(z\right)=
\sum\limits_{k=0}^{\infty}\frac{z^{k}}{\Gamma(\alpha k+\beta)},
\end{eqnarray*}
respectively, where $z,\beta\in\mathbb{C}$ and $\alpha>0$. The following Mittag-Leffler function  estimate is known \cite[Theorem 4]{Simon}:
\begin{eqnarray}\label{additive1.1}
\frac{1}{1+\Gamma\left(1-\alpha\right)z}\leq E_{\alpha,1}\left(-z\right)\leq\frac{1}{1+\Gamma\left(\alpha+1\right)^{-1}z}, 
\end{eqnarray}
and, in particular,  from this inequality it follows that
\begin{eqnarray}\label{additive1.2}
0<E_{\alpha,1}\left(-z\right)<1, 
\end{eqnarray}
for $0<\alpha<1$ and $z>0$.  Moreover,   we get  the  the estimate in the following form
\begin{eqnarray}\label{additive1.3}
\left|E_{\alpha,\beta}\left(-z\right)\right|\leq \frac{C}{1+z},
\end{eqnarray}
where $z>0$, $0<\alpha<2$, $\beta\in\mathbb{R}$ and $C>0$  (see \cite{Agarwal1953} and \cite[Subsection 2.1]{Podlubny1999}. These estimates of Mittag-Leffler functions have been widely applied to  the  solutions of certain problems formulated in terms of fractional order differential, integral and difference equations (see \cite{MH2008}, \cite{KST2006},    \cite{Podlubny1999} and \cite{Tarasov2010}).   Thus,  it has recently become a subject of interest for many authors in the field of fractional calculus and its applications on inverse source problems for the diffusion, sub-diffusion and for other types of equations.  Motivated by these avenues of applications, there are numerous works published only in recent years in this area (for details, see \cite{AKT19},  \cite{RTT19} and \cite{RSTT21}).

In \cite{Mansour2009},  Z. S. I.~Mansour  has introduced  a  $q$-analogue of the Mittag-Leffler function, and obtained a fundamental set of solutions for the homogeneous linear sequential $q$-difference equations with constant coefficients, and a general solution for the corresponding nonhomogeneous equations (see also \cite{AM2012}). The $q$-Mittag-Leffler function due to Mansour \cite{Mansour2009}, and  the  $q$-Mittag-Leffler function  
$e_{\alpha,\beta}\left(z;q\right)$   are given by
\begin{eqnarray}\label{additive1.4}
e_{\alpha,\beta}\left(z;q\right)=
\sum\limits_{k=0}^{\infty}\frac{z^{k}}{\Gamma_q(\alpha k+\beta)},\;\;\;\left|z(1-q)^\alpha\right|<1, 
\end{eqnarray}
when $\alpha>0$, $\beta\in \mathbb{C}$,  and $\Gamma_q$ is the $q$-Gamma function (see (\ref{additive2.9})).

In the present paper,   we establish estimates (\ref{additive1.1}),  (\ref{additive1.2}) and (\ref{additive1.3}) for   the function (\ref{additive1.4})  (see Lemmas \ref{lem3.1} and \ref{lem3.2}). Moreover, we apply these estimates to  prove  existence and uniqueness
of solutions of direct and inverse problems for time-fractional pseudo-parabolic equations with the  Caputo fractional operator in the  quantum calculus,  and  a self-adjoint operator in the abstract setting of Hilbert spaces.

This paper is organized as follows: In Section \ref{sc2}, we recall some necessary fundamental concepts of the quantum calculus.  In Section \ref{sc3}, we give some  lemmas to prove our main results. In Sections \ref{sc4} and \ref{sc5}, we apply these results to study the existence and uniqueness of  generalized solutions for the direct problems   and the inverse problems (see Theorem \ref{thm4.1}, Theorem \ref{thm4.2} and Theorem \ref{thm5.1}).

 \;\;\;

%%%%%%%%%%%%%%%%%%%%%%%%%%%%%%%%%%%%%%%%%%%%%%%%%%%%%%%%%%%%%%%
%%%%%%%%%%%%%%%%%%%%%%%%%%%%%%%%%%%%%%%%%%%%%%%%%%%%%%%%%%%%%%%

 \section{Preliminaries}\label{sc2} 
%%%%%%%%%%%%%%%%%%%%%%%%%%%%%%%%%%%%%%%%%%%%%%%%%%%%%%%%%%%%%%%
%%%%%%%%%%%%%%%%%%%%%%%%%%%%%%%%%%%%%%%%%%%%%%%%%%%%%%%%%%%%%%%

   For the convenience of the reader,  in this section, we recall some necessary concepts and definitions of the quantum calculus, for which we refer to \cite{ChKa2000, Ernst2002, GR1990} . Moreover,  we give  lemmas will be used to  prove of our main results. Throughout this paper, we assume that $0<q<1$.

Let $\alpha \in \mathbb{R}$. Then a $q$-real number $[\alpha ]_{q}$ is defined by
$$
[\alpha ]_{q}:=\frac{1-q^{\alpha }}{1-q},
$$
where $\mathop{\lim }\limits_{q\rightarrow 1}\frac{1-q^{\alpha }}{1-q}
=\alpha $.

For   $a\in\mathbb{R}$ and $n=0,1,\cdots$, the $q$-shifted factorials are defined by
 \begin{eqnarray}\label{additive2.1}
(a;q)_0=1, \;\;\; (a;q)_n=\prod\limits_{k=0}^n\left(1-q^ka\right), \;\;\;(a;q)_\infty&=&\lim\limits_{n\rightarrow\infty}(a;q)_n.
\end{eqnarray}

Note that   

 \begin{eqnarray}\label{additive2.2}
 (a;q)_{n}=\frac{(a;q)_\infty}{(aq^n;q)_\infty}, \;\;\;(a;q)_{2n}=(a;q^2)_n(aq;q^2)_n.
 \end{eqnarray}

The  function $e_q(x)$  is a $q$-exponential function  defined by its $q$-Taylor series (\cite{ChKa2000}):
\begin{eqnarray}\label{additive2.3}
e_q(x)=\sum\limits_{k=0}^\infty\frac{x^{k}}{[k]_{q}!}, 
\end{eqnarray}
where 
\begin{equation*}
[k]_{q}!=\left\{
\begin{array}{l}
{1,\mathrm{\;\;\;\;\;\;\;\;\;\;\;\;\;\;\;\;\;\;\;\;\;\;\;\;\;\;\;\;\;\;\;\;%
\;if\;{\it k}}=\mathrm{0,}} \\
{[1]_{q}\times [2]_{q}\times \cdots \times [n]_{q},\mathrm{%
\;if\;{\it k}}\in \mathrm{ \mathbb{N}.\;\;}}%
\end{array} \right.
\end{equation*}

Recall that the $q$-difference operator (or the $q$-derivative), acting on functions of the  variable $x$, $D_q$ is defined by 
\begin{eqnarray}\label{additive2.4}
D_{q}f(x)=\frac{f(x)-f(qx)}{x(1-q)}.
\end{eqnarray}
Note that if $f$ is differentiable at $x$, then
$\lim\limits_{q\rightarrow1}D_{q}f(x)=f'(x)$, and the operator $D_q$ is called  the Euler–Jackson difference operator (see \cite{Jackson1908}). Note that
\begin{eqnarray}\label{additive2.5}
D_{q}\left[x^\alpha(s/x;q)_\alpha\right]=-[\alpha]_qx^{\alpha-1}(qs/x;q)_{\alpha-1}, \;\;\;\alpha\in\mathbb{R},
\end{eqnarray}
where
\begin{eqnarray*} 
a^\alpha(s/a;q)_\alpha=a^\alpha\frac{(s/a;q)_\infty}{(q^{\alpha}s/a;q)_\infty}. 
\end{eqnarray*}

The $q$-translation operator $\varepsilon^y$  is presented by M. E. H.~Ismail in \cite{Ismail2005} and is defined  by
 \begin{eqnarray*} 
 \varepsilon^yx^\alpha=x^\alpha(y/x;q)_\alpha,
\end{eqnarray*}
and its extension is given as follows (see, \cite[Section 1.2]{AM2012}): 
\begin{eqnarray}\label{additive2.6}
\varepsilon^y\left(\sum\limits_{k=0}^\infty\alpha_kx^\alpha\right)= \sum\limits_{k=0}^\infty \alpha_kx^\alpha(y/x;q)_\alpha.
\end{eqnarray}

Therefore, using (\ref{additive2.6}) we can take a more general form of (\ref{additive1.4}) as follows 
\begin{eqnarray}\label{additive2.7}
\varepsilon^ye_{\alpha,\beta}\left(z^\gamma;q\right)=
\sum\limits_{k=0}^{\infty}\frac{x^{\alpha{k}}(y/x;q)_{\alpha{k}}}{\Gamma_q(\alpha k+\beta)}.  
\end{eqnarray}

The $q$-integral (or Jackson integral) is defined as (see \cite{Jackson1910})
\begin{eqnarray*}
\int\limits_0^a f(x)d_{q}x=(1-q)a\sum\limits_{m=0}^\infty q^{m}f(aq^{m})
\end{eqnarray*}
and
\begin{eqnarray*} 
\int\limits_0^\infty f(x)d_{q}x=(1-q)\sum\limits_{m=-\infty}^\infty q^{m}f(q^{m}),
\end{eqnarray*}
provided that the sums converge absolutely. The $q$-Jackson integral in the  interval $[a, b]$ is defined by
\begin{eqnarray*}
\int\limits_a^b f(x)d_{q}x=\int\limits_0^b f(x)d_{q}x-
\int\limits_0^a f(x)d_{q}x.
\end{eqnarray*}

The formula for the $q$–integration by parts is
\begin{eqnarray}\label{additive2.8}
\int\limits_a^b u(x)D_qv(x)d_qx=\left[u(x)v(x)\right]_a^b-\int\limits_a^bv(qx)D_qu(x)d_qx.
\end{eqnarray}

For any $x>0$, the $q$–gamma function is defined by 
\begin{eqnarray}\label{additive2.9}
\Gamma_q(x)=\frac{(q;q)_\infty }{(q^x; q)_\infty}(1-q)^{1-x},\;\;\;\;\Gamma_q\left(x+1\right)=\Gamma_q\left(x\right)\left[x\right]_q.
\end{eqnarray}

The  Riemann-Liouville $q$-fractional integral $I_{q,a+}^{\alpha }f$ of order $\alpha > 0$ is defined by (\cite{Raj})
\begin{equation*} 
  \left( I_{q,a+}^{\alpha }f \right)\left( x \right)=\frac{1}{{{\Gamma }_{q}}\left( \alpha  \right)}\int\limits_{a}^{x}{{{\left( x-qt \right)}^{\alpha -1}_{q}}f\left( t \right){{d}_{q}}t}.  
\end{equation*}

The fractional $q$-derivative of Caputo type of order $\alpha > 0$ is defined by (\cite{Raj})
\begin{equation}\label{additive2.10}
  \left(^cD_{q,a+}^{\alpha }f \right)\left( x \right)=\left(I_{q,a+}^{\left[ \alpha  \right]-\alpha } D_{q,a+}^{\left[ \alpha  \right]}f \right)\left( x \right),  
\end{equation}
where $[\alpha]$ denotes the smallest integer greater or equal to $\alpha$.

\textbf{Notation.}  The symbol $M \lesssim K$ means that there exists $\gamma> 0$ such that
$M \leq \gamma K$, where $\gamma$ is a constant.   If $M \lesssim K \lesssim  M$, then we
write $M \approx K$.

\;\;\;\;

%%%%%%%%%%%%%%%%%%%%%%%%%%%%%%%%%%%%%%%%%%%%%%%%%%%%%%%%%%%%%%%
%%%%%%%%%%%%%%%%%%%%%%%%%%%%%%%%%%%%%%%%%%%%%%%%%%%%%%%%%%%%%%%
\section{ Some estimates of  the  $q$-Mittag-Leffler function }\label{sc3}
%%%%%%%%%%%%%%%%%%%%%%%%%%%%%%%%%%%%%%%%%%%%%%%%%%%%%%%%%%%%%%%
%%%%%%%%%%%%%%%%%%%%%%%%%%%%%%%%%%%%%%%%%%%%%%%%%%%%%%%%%%%%%%
In this section we establish estimates (\ref{additive1.1}),  (\ref{additive1.2}) and (\ref{additive1.3}) for   the function (\ref{additive1.4}) in the following lemmas.

\begin{lem}\label{lem3.1} If $0<\alpha <1$ and $z >0$. Then
\begin{eqnarray}\label{additive3.1}
\frac{1}{1+\Gamma_q\left(1-\alpha\right)z}\leq e_{\alpha,1}\left(-z;q\right)\leq\frac{1}{1+\Gamma_q\left(\alpha+1\right)^{-1}z}.
\end{eqnarray}

In particular, we have
\begin{eqnarray}\label{additive3.2}
0<e_{\alpha,1}\left(-z;q\right)<1.
\end{eqnarray}
\end{lem} 
 \begin{proof} {\it The lower estimate}.  We assume that  $0<\alpha <1$. Then  using (\ref{additive2.1}) we get that 
\begin{eqnarray}\label{additive3.3}
q/q^{1-\alpha}=q^{\alpha}\leq1\Rightarrow q\leq q^{1-\alpha} 
&\Rightarrow& 1-q^{1-\alpha}\leq1-q  \nonumber\\
&\Rightarrow&\prod\limits_{n=0}^\infty(1-q^nq^{1-\alpha})\leq\prod\limits_{n=0}^\infty(1-q^nq) \nonumber\\
&\overset{ \text{(\ref{additive2.1})}}{\Rightarrow}&\gamma_{1,q}:=\frac{\left(q;q\right)_\infty}{\left(q^{1-\alpha};q\right)_\infty}\geq 1;
\end{eqnarray}
\begin{eqnarray}\label{additive3.4}
q^k/q^{k\alpha}=q^{k(1-\alpha)}\leq1\Rightarrow q^k\leq q^{k\alpha}&\Rightarrow&1-q^{\alpha{k}+1}\leq1-q^{k+1} \nonumber\\
&\overset{ \text{(\ref{additive2.1})}}{\Rightarrow}&\gamma_{2,q}:=\frac{\left(q^{\alpha{k}+1};q\right)_\infty}{\left(q^{k+1};q\right)_\infty} \leq 1;
\end{eqnarray}
\begin{eqnarray}\label{additive3.5}
q^{k+1}/q^{\alpha{k}+1-\alpha}=q^{k(1-\alpha)+\alpha}\leq1&\Rightarrow& q^{k+1}\leq q^{\alpha{k}+1-\alpha}\nonumber\\
&\Rightarrow&1-q^{\alpha{k}+1-\alpha}\leq1-q^{k+1}\nonumber\\
&\overset{ \text{(\ref{additive2.1})}}{\Rightarrow}&\gamma_{3,q}:=\frac{\left(q^{\alpha{k}+1-\alpha};q\right)_\infty}{\left(q^{k+1};q\right)_\infty}\leq 1,
\end{eqnarray}
 for   $k=0,1,2,3,\cdots$.

 Applying (\ref{additive2.9}) and  these inequalities (\ref{additive3.3})-(\ref{additive3.5}), we find that 
\begin{eqnarray*}  
N_q(k)&:=& \frac{\left(q^{\alpha{k}+1};q\right)_\infty}{\left(q;q\right)_\infty}(1-q)^{\alpha{k}}- \frac{\left(q^{\alpha(k-1)+1};q\right)_\infty}{\left(q^{1-\alpha};q\right)_\infty}(1-q)^{\alpha{k}} \nonumber\\
&\overset{ \text{(\ref{additive3.3})(\ref{additive3.4})}}{=}&(1-q)^{\alpha{k}}\left[\gamma_{2,q} 
\frac{\left(q^{k+1};q\right)_\infty}{\left(q;q\right)_\infty} - \gamma_{1,q} 
\frac{\left(q^{\alpha(k-1)+1};q\right)_\infty}{\left(q;q\right)_\infty}\right] \nonumber\\
&\leq&(1-q)^{\alpha{k}}\left[\frac{\left(q^{k+1};q\right)_\infty}{\left(q;q\right)_\infty}- \frac{\left(q^{\alpha{k}+1-\alpha};q\right)_\infty}{\left(q;q\right)_\infty}  \right]\nonumber\\ 
&\overset{ \text{(\ref{additive3.5})}}{=}&\frac{\left(q^{k+1};q\right)_\infty}{\left(q;q\right)_\infty}(1-q)^{\alpha{k}}\left[1- \gamma_{3,q}   \right]
\end{eqnarray*}
\begin{eqnarray}\label{additive3.6}
&\leq&(1-q)^{\alpha{k}-k} \left[(1-q)^{1-(k+1)}\frac{\left(q;q\right)_\infty}{\left(q^{k+1};q\right)_\infty}\right]^{-1} \nonumber\\
&\overset{ \text{(\ref{additive2.9})} }{=}&\frac{(1-q)^{(\alpha-1)k}}{\Gamma_q\left(k+1\right)}.
\end{eqnarray}

Thus, 
\begin{eqnarray}\label{additive3.7}
N_q(k)&\lesssim&\frac{(1-q)^{(\alpha-1)k}}{\Gamma_q\left(k+1\right)}.
\end{eqnarray}

We assume that $0<\epsilon<1-\alpha$. Then using (\ref{additive2.1}) we write the following inequalities 
\begin{eqnarray}\label{additive3.8}
q^{{\alpha}k+\alpha}/q^{{\alpha}k+1}=q^{\alpha-1}\geq1 \Rightarrow q^{{\alpha}k+1}\leq q^{{\alpha}k+\alpha}&\Rightarrow&1-q^{{\alpha}k+\alpha}\leq1-q^{{\alpha}k+1}\nonumber\\
&\overset{ \text{(\ref{additive2.1})}}{\Rightarrow}& \left(q^{{\alpha}k+\alpha};q\right)_\infty\leq \left(q^{{\alpha}k+1};q\right)_\infty; 
 \end{eqnarray}
\begin{eqnarray}\label{additive3.9}
q^{k+1}/q^{{\alpha}k+1-\alpha}=q^{k(1-\alpha)+\alpha}\leq1&\Rightarrow& q^{k+1}\leq q^{{\alpha}k+1-\alpha}\nonumber\\
&\overset{ \text{(\ref{additive2.1})}}{\Rightarrow}& \left(q^{{\alpha}k+1-\alpha};q\right)_\infty\leq \left(q^{k+1};q\right)_\infty;
 \end{eqnarray}
 \begin{eqnarray}\label{additive3.10}
q/q^{1-\alpha}=q^{\alpha}\leq1&\Rightarrow &q \leq q^{1-\alpha} \nonumber\\
 &\overset{ \text{(\ref{additive2.1})}}{\Rightarrow}&   \left(q^{1-\alpha};q\right)_\infty\leq\left(q;q\right)_\infty; 
 \end{eqnarray}
\begin{eqnarray*}
\left.\begin{array}{rcl}
 q^{1-\alpha}/ (q^\epsilon/3)= 3 q^{1-\alpha-\epsilon}\leq3\Rightarrow q^{1-\alpha}\leq 3(q^\epsilon/3) &\overset{ \text{(\ref{additive2.1})}}{\Rightarrow}3\leq \frac{\left(q^{1-\alpha};q\right)_\infty}{\left(q^\epsilon/3;q\right)_\infty} \\
q^{k+1}/q^{\alpha{k}+\alpha}=q^{-\alpha(k+1)}\geq1   \Rightarrow   q^{\alpha{k}+\alpha}\leq q^{k+1}&\overset{ \text{(\ref{additive2.1})}}{\Rightarrow} 1\leq \frac{\left(q^{\alpha{k}+\alpha};q\right)_\infty }{\left(q^{k+1};q\right)_\infty}
\end{array}\right\}\Rightarrow
\end{eqnarray*}
\begin{eqnarray}\label{additive3.11}
\Rightarrow \gamma_{4,q}:= \frac{\left(q^{1-\alpha};q\right)_\infty}{\left(q^\epsilon/3;q\right)_\infty}\frac{\left(q^{\alpha{k}+\alpha};q\right)_\infty }{\left(q^{k+1};q\right)_\infty}\geq 3.
 \end{eqnarray}

Let $q_0:=q^\epsilon/3$ and $\gamma_{5,q}:=\frac{1-q_0^{1-\epsilon\alpha}}{q^{\epsilon(1-\epsilon\alpha)}(1-q_0)}$. Then $\lim\limits_{\epsilon\rightarrow0} \gamma_{5,q}=1$ and 
 \begin{eqnarray}\label{additive3.12}
qq_0/q=q_0\leq1\Rightarrow qq_0\leq q
&\Rightarrow&  1-q \leq 1-qq_0\nonumber\\
&\overset{ \text{(\ref{additive2.1})}}{\Rightarrow}&  \left(q ;q\right)_\infty\leq\left(qq_0;q\right)_\infty, 
 \end{eqnarray}
and applying above inequalities and  (\ref{additive2.1}) and  (\ref{additive3.8})-(\ref{additive3.12}) we have that
\begin{eqnarray*}   
\widehat{N}_q(k)&:=& \gamma_{5,q}\frac{\left(q^{\alpha{k}+1};q\right)_\infty}{\left(q;q\right)_\infty}- q^\epsilon \frac{\left(q^{\alpha(k-1)+1};q\right)_\infty}{\left(q^{1-\alpha};q\right)_\infty}\nonumber\\
&\overset{ \text{ (\ref{additive3.8})(\ref{additive3.9})(\ref{additive3.12})} }{\geq}& \frac{1-q_0^{1-\epsilon\alpha}}{q^{\epsilon(1-\epsilon\alpha)}(1-q_0)}\frac{\left(q^{\alpha{k}+\alpha};q\right)_\infty}{\left(qq_0;q\right)_\infty}- q^\epsilon \frac{\left(q^{k+1};q\right)_\infty}{\left(q^{1-\alpha};q\right)_\infty}\nonumber\\
&\overset{ \text{(\ref{additive2.1})} }{=}&   \frac{1-q_0^{1-\epsilon\alpha}}{q^{\epsilon(1-\epsilon\alpha)} }\frac{\left(q^{\alpha{k}+\alpha};q\right)_\infty}{\left(q_0;q\right)_\infty}-  q^\epsilon \frac{\left(q^{k+1};q\right)_\infty}{\left(q^{1-\alpha};q\right)_\infty} \nonumber\\
&=& \frac{\left(q^{k+1};q\right)_\infty}{\left(q^{1-\alpha};q\right)_\infty} \left[\frac{ 1-3^{\epsilon\alpha-1}q^{\epsilon(1-\epsilon\alpha)} }{q^{\epsilon(1-\epsilon\alpha)}}\gamma_{4,q}- q^\epsilon  \right]
\end{eqnarray*}
\begin{eqnarray}\label{additive3.13}
&\overset{ \text{(\ref{additive3.11})} }{\geq}& \frac{\left(q^{k+1};q\right)_\infty}{\left(q^{1-\alpha};q\right)_\infty} \left[3\frac{ (1-3^{\epsilon\alpha-1}q^{\epsilon(1-\epsilon\alpha)})}{q^{\epsilon(1-\epsilon\alpha)}} - q^\epsilon  \right]\nonumber\\
&\overset{ \text{(\ref{additive3.10})} }{\geq}& \frac{\left(q^{k+1};q\right)_\infty}{\left(q ;q\right)_\infty}3\left[q^{-\epsilon(1-\epsilon\alpha)}-3^{\epsilon\alpha-1}    - q^\epsilon/3 \right].
\end{eqnarray}

From (\ref{additive2.9}), (\ref{additive3.6})  and  (\ref{additive3.13}), we get that
\begin{eqnarray*}  
N_q(k)&\overset{ \text{(\ref{additive3.6})} }{=}&(1-q)^{\alpha{k}}\left[\frac{\left(q^{\alpha{k}+1};q\right)_\infty}{\left(q;q\right)_\infty}- \frac{\left(q^{\alpha(k-1)+1};q\right)_\infty}{\left(q^{1-\alpha};q\right)_\infty}\right]\\
&\overset{ \text{(\ref{additive3.13})} }{=}&(1-q)^{\alpha{k}}\lim\limits_{\epsilon\rightarrow0}\widehat{N}_q(k) \\
&\geq& (1-q)^{\alpha{k}} \frac{\left(q^{k+1};q\right)_\infty}{\left(q;q\right)_\infty} 3\left[1-1/3- 1/3 \right]\\
&\geq&(1-q)^{\alpha{k}-k} \left[(1-q)^{1-(k+1)}\frac{\left(q;q\right)_\infty}{\left(q^{k+1};q\right)_\infty}\right]^{-1}\\
&\overset{ \text{(\ref{additive2.9})} }{=}&\frac{(1-q)^{(\alpha-1)k}}{\Gamma_q\left(k+1\right)},
 \end{eqnarray*}
for   $k=0,1,2,3,\cdots$. 
 
Hence,
\begin{eqnarray}\label{additive3.14}
\frac{(1-q)^{(\alpha-1)k}}{\Gamma_q\left(k+1\right)}\lesssim N_q(k).
\end{eqnarray}

By  using (\ref{additive1.4}), (\ref{additive3.6}), (\ref{additive3.7}) and  (\ref{additive3.14})  we find
\begin{eqnarray*} 
\frac{e_{\alpha,1}\left(-z;q\right)}{\left[1+ \Gamma_q\left(1-\alpha\right)z\right]^{-1}}&=&e_{\alpha,1}\left(-z;q\right)+e_{\alpha,1}\left(-z;q\right)\Gamma_q\left(1-\alpha\right)z\nonumber\\&\overset{ \text{(\ref{additive1.4})}}{=}&\sum\limits_{k=0}^\infty\frac{\left(-z\right)^k}{\Gamma_q\left(\alpha{k}+1\right)}
+\sum\limits_{k=0}^\infty\frac{\left(-z\right)^{k+1}\Gamma_q\left(1-\alpha\right)}{\Gamma_q\left(\alpha{k}+1\right)}\nonumber\\
&=&1+\sum\limits_{k=0}^\infty\frac{\left(-z\right)^k}{\Gamma_q\left(\alpha{k}+1\right)}
-\left[1+ \sum\limits_{k=1}^\infty\frac{\left(-z\right)^k\Gamma_q\left(1-\alpha\right)}{\Gamma_q\left(\alpha(k-1)+1\right)}\right]\nonumber\\
&=&1+\sum\limits_{k=0}^\infty\frac{\left(-z\right)^k}{\Gamma_q\left(\alpha{k}+1\right)} - \sum\limits_{k=0}^\infty\frac{\left(-z\right)^k\Gamma_q\left(1-\alpha\right)}{\Gamma_q\left(\alpha(k-1)+1\right)} \nonumber\\
&=&1+\sum\limits_{k=0}^\infty\left(-z\right)^k\left[\frac{1}{\Gamma_q\left(\alpha{k}+1\right)} -  \frac{ \Gamma_q\left(1-\alpha\right)}{\Gamma_q\left(\alpha(k-1)+1\right)} \right]\nonumber\\
&\overset{ \text{(\ref{additive2.1})}}{=}&1+\sum\limits_{k=0}^\infty\left(-z\right)^k\left[\frac{\left(q^{\alpha{k}+1};q\right)_\infty}{\left(q;q\right)_\infty}(1-q)^{\alpha{k}}\right.\nonumber\\
&-&\left. \frac{\left(q^{\alpha(k-1)+1};q\right)_\infty}{\left(q^{1-\alpha};q\right)_\infty}(1-q)^{\alpha{k}} \right]
\end{eqnarray*}
\begin{eqnarray*}
&\overset{ \text{(\ref{additive3.6})}}{=}&1+\sum\limits_{k=0}^\infty\left(-z\right)^kN_q(k)\nonumber\\
&\overset{ \text{(\ref{additive3.7})(\ref{additive3.14})}}{ \approx}&1+\sum\limits_{k=0}^\infty\left(-z\right)^k\frac{(1-q)^{\alpha{k}-k}}{\Gamma_q\left(k+1\right)}\nonumber\\
&\overset{ \text{(\ref{additive2.9})}}{=}&1+\sum\limits_{k=0}^\infty\left(-\frac{z}{(1-q)^{1-\alpha}}\right)^k\frac{1}{[k]_q!} \nonumber\\
&\overset{ \text{(\ref{additive2.3})}}{=}&1+ e_q{\left(-\frac{z}{(1-q)^{1-\alpha}}\right)} \nonumber\\
&\geq&1.
\end{eqnarray*}

Therefore
\begin{eqnarray}\label{additive3.15}
0<\frac{1}{1 +\Gamma_q\left(1-\alpha\right)z}\leq e_{\alpha,1}\left(-z;q\right),
\end{eqnarray}
and the lower estimates in (\ref{additive3.1}) and (\ref{additive3.2}) are proved.

{\it The upper estimate}. By using (\ref{additive2.1}) for $0<\alpha <1$ and $k\in\mathbb{N}$,  we have the following inequalities
\begin{eqnarray}\label{additive3.16}
q^{\alpha+1}/q=q^\alpha\leq1\Rightarrow1-q\leq1-q^{\alpha+1}&\Rightarrow&\prod\limits_{n=0}^\infty\left(1-q^{n+1}\right)\leq\prod\limits_{n=0}^\infty\left(1-q^{\alpha+1+n}\right)\nonumber\\
&\overset{ \text{(\ref{additive2.1})}}{\Rightarrow}& 1\leq\frac{\left(q^{\alpha+1};q\right)_\infty}{ \left(q;q\right)_\infty};
\end{eqnarray}
\begin{eqnarray}\label{additive3.17}
q^{\alpha(k-1)+1}/q=q^{\alpha(k-1)}\leq1{\Rightarrow}q^{\alpha(k-1)+1}\leq{q}&\Rightarrow&1-q\leq1-q^{\alpha(k-1)+1} \nonumber\\
&\overset{ \text{(\ref{additive2.1})}}{\Rightarrow}& 1\leq\frac{\left(q^{\alpha(k-1)+1};q\right)_\infty}{ \left(q;q\right)_\infty};
\end{eqnarray}
 \begin{eqnarray}\label{additive3.18}
q^{\alpha{k}+1}/q^{k+1}=q^{\alpha -1}\geq1{\Rightarrow}q^{\alpha{k}+1}\geq{q^{k+1}}&\Rightarrow&1-q^{\alpha{k}+1}\leq1-q^{k+1} \nonumber\\
&\overset{ \text{(\ref{additive2.1})}}{\Rightarrow}& 1\geq\frac{\left(q^{\alpha{k}+1};q\right)_\infty}{ \left(q^{k+1};q\right)_\infty}.
\end{eqnarray}

We assume that $0<\xi<\alpha$. Then using the inequality $\left(q^{k+1};q\right)_\infty\leq1$, for $k\in\mathbb{N}$,    and (\ref{additive3.16}) - (\ref{additive3.18}), we find that 
\begin{eqnarray}\label{additive3.19}
\gamma_{6,q}&:=&\frac{1}{1-\xi}\frac{(1-q)^{\alpha{k}}}{\left(\xi{q};q\right)_\infty} \frac{\left(q^{\alpha+1};q\right)_\infty}{\left(q;q\right)_\infty} \frac{\left(q^{\alpha(k-1)+1};q\right)_\infty}{\left(q ;q\right)_\infty}\nonumber\\
&\overset{ \text{(\ref{additive3.16})(\ref{additive3.17})}}{\geq}&\frac{1}{1-\xi} \frac{(1-q)^{\alpha{k}}}{\left(\xi{q};q\right)_\infty} \nonumber\\
&\geq & \frac{(1-q)^{\alpha{k}} }{1-\xi}\frac{\left(q^{k+1};q\right)_\infty}{\left(q;q\right)_\infty}\nonumber\\
&\geq & \frac{(1-q)^{\alpha{k}} }{1-\alpha}\frac{\left(q^{k+1};q\right)_\infty}{\left(q;q\right)_\infty},
\end{eqnarray}
where $\lim\limits_{\xi\rightarrow0}\frac{q^\xi}{\left(\xi{q};q\right)_\infty}=1$ and 
\begin{eqnarray}\label{additive3.20}
\gamma_{7,q}&:=& q^\xi(1-q)^{\alpha{k}}\frac{\left(q^{\alpha{k}+1};q\right)_\infty}{\left(q;q\right)_\infty}\nonumber\\
&=& q^\xi(1-q)^{\alpha{k}}\frac{\left(q^{k+1};q\right)_\infty}{\left(q;q\right)_\infty} \frac{\left(q^{\alpha{k}+1};q\right)_\infty}{\left(q^{k+1};q\right)_\infty}\nonumber\\
&\overset{(\ref{additive3.18})}{\leq}& q^\xi (1-q)^{\alpha{k}}\frac{\left(q^{k+1};q\right)_\infty}{\left(q;q\right)_\infty}.
\end{eqnarray}

From (\ref{additive2.1}), (\ref{additive2.9}) and   (\ref{additive3.19})-(\ref{additive3.20}) we have that  
\begin{eqnarray}\label{additive3.21}
M_q(k)&:=&\frac{\Gamma_q^{-1}\left(\alpha+1\right)}
{ \Gamma_q\left(\alpha(k-1)+1\right)}-\frac{1}{2\Gamma_q\left(\alpha{k}+1\right)}\nonumber\\
&\overset{ \text{(\ref{additive2.1})}}{=}&(1-q)^{\alpha{k}}\frac{\left(q^{\alpha+1};q\right)_\infty}{\left(q;q\right)_\infty}\frac{\left(q^{\alpha(k-1)+1};q\right)_\infty}{\left(q;q\right)_\infty} -  (1-q)^{\alpha{k}} \frac{\left(q^{\alpha{k}+1};q\right)_\infty}{\left(q;q\right)_\infty}  \nonumber\\
& \overset{ \text{(\ref{additive3.19})(\ref{additive3.20})}}{=}&\lim\limits_{\xi\rightarrow0}\left[\gamma_{6,q}-\gamma_{7,q}\right] \nonumber\\
&\geq&(1-q)^{\alpha{k}}\frac{\left(q^{k+1};q\right)_\infty}{\left(q;q\right)_\infty}\lim\limits_{\xi\rightarrow0}\left[\frac{1}{1-\alpha}-q^\xi\right] \nonumber\\
&=& (1-q)^{\alpha{k}-k} \left[(1-q)^{(k+1)-1}\frac{\left(q^{k+1};q\right)_\infty}{\left(q;q\right)_\infty}\right]\left[\frac{1}{1-\alpha}-1\right]\nonumber\\
&\overset{ \text{(\ref{additive2.9})}}{=}&\frac{\alpha}{1-\alpha}\frac{(1-q)^{\alpha{k}-k}}{\Gamma_q\left(k+1\right)}.
\end{eqnarray}

Thus, 
\begin{eqnarray}\label{additive3.22}
 \frac{(1-q)^{\alpha{k}-k}}{\Gamma_q\left(k+1\right) } &\lesssim&M_q(k).
\end{eqnarray}

Again, for  $k\in\mathbb{N}$ using (\ref{additive2.1}), we find the following inequalities 
\begin{eqnarray}\label{additive3.23}
q^{\alpha{k}+1}/q^{\alpha(k-1)+1}=q^{-\alpha}\geq1 &\Rightarrow& 1-q^{\alpha(k-1)+1}  \leq  1-q^{\alpha{k}+1} \nonumber\\ 
&\overset{ \text{(\ref{additive2.1})}}{\Rightarrow}&
 \left(q^{\alpha(k-1)+1};q\right)_\infty\le\left(q^{\alpha{k}+1};q\right)_\infty;
\end{eqnarray}
\begin{eqnarray}\label{additive3.24}
q^{\alpha(k-1)+1}/q^{k+1}=q^{-k-\alpha+\alpha{k}}\geq1 &\Rightarrow& 1-q^{\alpha(k-1)+1} \leq  1-q^{k+1} \nonumber\\
&\overset{ \text{(\ref{additive2.1})}}{\Rightarrow}& \left(q^{\alpha(k-1)+1};q\right)_\infty \leq\left(q^{k+1};q\right)_\infty;
\end{eqnarray}
\begin{eqnarray}\label{additive3.25}
q^{\alpha}/q=q^{\alpha-1}\geq1 &\Rightarrow& 1-q^{\alpha} \leq  1-q \nonumber\\
&\overset{ \text{(\ref{additive2.1})}}{\Rightarrow}&  \frac{\left(q^{\alpha};q\right)_\infty}{\left(q;q\right)_\infty}\leq 1.
\end{eqnarray}

First, we consider the case $q^\alpha\leq1-q^\alpha<1$. Then $ \frac{q^\alpha}{1-q^\alpha} \leq1 $.  Form (\ref{additive2.1}), (\ref{additive3.21})  and (\ref{additive3.23})-(\ref{additive3.25}),  we have that 
that 
\begin{eqnarray}\label{additive3.26}
M_q(k)&\overset{ \text{(\ref{additive3.21})}}{=}&(1-q)^{\alpha{k}}\frac{\left(q^{\alpha+1};q\right)_\infty}{\left(q;q\right)_\infty}\frac{\left(q^{\alpha(k-1)+1};q\right)_\infty}{\left(q;q\right)_\infty} -  (1-q)^{\alpha{k}} \frac{\left(q^{\alpha{k}+1};q\right)_\infty}{\left(q;q\right)_\infty} \nonumber\\
&\overset{ \text{(\ref{additive3.23})}}{\leq}&(1-q)^{\alpha{k}}\left[ \frac{\left(q^{\alpha+1};q\right)_\infty}{\left(q;q\right)_\infty}\frac{\left(q^{\alpha(k-1)+1};q\right)_\infty}{\left(q;q\right)_\infty} -   \frac{\left(q^{\alpha{k-1}+1};q\right)_\infty}{\left(q;q\right)_\infty}\right]\nonumber\\
&\leq&(1-q)^{\alpha{k}}\frac{\left(q^{\alpha(k-1)+1};q\right)_\infty}{\left(q;q\right)_\infty}\left[\frac{1-q^\alpha}{1-q^\alpha}\frac{\left(q^{\alpha+1};q\right)_\infty}{\left(q;q\right)_\infty} - 1\right]\nonumber\\
&\overset{ \text{(\ref{additive2.1})(\ref{additive3.24})}}{=}&(1-q)^{\alpha{k}}\frac{\left(q^{k+1};q\right)_\infty}{\left(q;q\right)_\infty}\left[\frac{1}{1-q^\alpha}\frac{\left(q^{\alpha};q\right)_\infty}{\left(q;q\right)_\infty} - 1\right]\nonumber\\
&\overset{ \text{(\ref{additive3.25})}}{=}&(1-q)^{\alpha{k}}\frac{\left(q^{k+1};q\right)_\infty}{\left(q;q\right)_\infty}\left[\frac{1}{1-q^\alpha} - 1\right]\nonumber\\
&\leq&(1-q)^{\alpha{k}-k}\left[(1-q)^{(k+1)-1}\frac{\left(q^{k+1};q\right)_\infty}{\left(q;q\right)_\infty}\right]\left[\frac{1}{1-q^\alpha} - \frac{q^\alpha}{1-q^\alpha}\right]\nonumber\\
&\overset{ \text{(\ref{additive2.9})}}{=}& \frac{(1-q)^{\alpha{k}-k}}{\Gamma_q\left(k+1\right)}.  
\end{eqnarray}

Next, we the case $1-q^\alpha<q^\alpha$. In fact $1/2<q^\alpha$. Using (\ref{additive3.26})   we write that 
\begin{eqnarray*} 
M_q(k) 
&\leq&(1-q)^{\alpha{k}}\frac{\left(q^{k+1};q\right)_\infty}{\left(q;q\right)_\infty}\left[\frac{1}{1-q^\alpha} - 1\right]\nonumber\\
&\leq&(1-q)^{\alpha{k}}\frac{\left(q^{k+1};q\right)_\infty}{\left(q;q\right)_\infty}\left[\frac{1}{1-1/2} - 1\right]\nonumber\\
&=& \frac{(1-q)^{\alpha{k}-k}}{\Gamma_q\left(k+1\right)}.  
\end{eqnarray*}

Therefore,  
\begin{eqnarray}\label{additive3.27}
M_q(k)&\lesssim&  \frac{(1-q)^{\alpha{k}-k}}{\Gamma_q\left(k+1\right)}. 
\end{eqnarray}

According  (\ref{additive2.3}), (\ref{additive2.9}), (\ref{additive3.22})  and (\ref{additive3.27})  we find that 
\begin{eqnarray*} 
\frac{e_{\alpha,1}\left(-z;q\right)}{\left[1+ \Gamma^{-1}_q\left(\alpha+1\right)z\right]^{-1}}&=&e_{\alpha,1}\left(-z;q\right)+ \Gamma^{-1}_q\left(\alpha+1\right)ze_{\alpha,1}\left(-z;q\right)\nonumber\\
&\overset{ \text{(\ref{additive2.9})}}{=}& \sum\limits_{k=0}^\infty\frac{\left(-z\right)^k}{\Gamma_q\left(\alpha{k}+1\right)} 
+ \sum\limits_{k=0}^\infty\frac{\left(-z\right)^{k+1}\Gamma^{-1}_q\left(\alpha+1\right)}{\Gamma_q\left(\alpha{k}+1\right)} \nonumber\\
&=&1- \sum\limits_{k=1}^\infty\frac{\left(-z\right)^k}{\Gamma_q\left(\alpha{k}+1\right)} 
- \sum\limits_{k=1}^\infty\frac{\left(-z\right)^k\Gamma^{-1}_q\left(\alpha+1\right)}{\Gamma_q\left(\alpha(k-1)+1\right)} 
\end{eqnarray*}
\begin{eqnarray*}
&=&1- \sum\limits_{k=1}^\infty \left(-z\right)^k \left[\frac{\Gamma_q^{-1}\left(\alpha+1\right)}
{ \Gamma_q\left(\alpha(k-1)+1\right)}-\frac{1}{ \Gamma_q\left(\alpha{k}+1\right)}\right] \nonumber\\
&\overset{ \text{(\ref{additive3.21})}}{=}&1-\sum\limits_{k=1}^\infty(-z)^kM_q(k)\nonumber\\
&\overset{ \text{(\ref{additive3.22})(\ref{additive3.27})}}{\approx}&  1-\sum\limits_{k=0}^\infty\frac{\left(-\frac{z}{(1-q)^{1-\alpha}}\right)^k}{[k]_q!} \nonumber\\
&\overset{ \text{(\ref{additive2.3})}}{=}&  1-e_q\left(-\frac{z}{(1-q)^{1-\alpha}}\right)  \nonumber\\
&\leq&1.
\end{eqnarray*}

Hence,
\begin{eqnarray*} 
e_{\alpha,1}\left(-z;q\right)&\leq&\frac{1}{1+ \Gamma^{-1}_q\left(\alpha+1\right)z}\nonumber\\
&\leq&\frac{1}{1+\Gamma^{-1}_q\left(\alpha+1\right)z}\nonumber\\
&\leq&1,
\end{eqnarray*}
which implies  that also the
upper estimates in (\ref{additive3.1}) and (\ref{additive3.2}) are proved. This completes the proof. 
\end{proof}

\begin{lem}\label{lem3.2} Let $0<\alpha<2$,   $0\leq\beta$ and $z>0$.   Then we get 
\begin{eqnarray}\label{additive3.28}
\left|e_{\alpha,\beta}\left(-z;q\right)\right|\leq \frac{C_q}{1+z},
\end{eqnarray}
where $C_q:= \frac{e_{q^{1/2}}(\alpha')}{\Gamma_q(\beta)}$ and $\alpha':=\frac{2^\alpha}{\left(1-q^\beta\right)(1-q^{1/2})}$.
\end{lem}

\begin{proof}  Let $0<\alpha<2$,   $0\leq\beta$ and $k\in\mathbb{N}$. From (\ref{additive2.1}), we get the following inequalities
\begin{eqnarray}\label{additive3.29}
q^{2{k}+\beta}\leq q^{\alpha{k}+\beta}\leq q^{\alpha{(k-1)}+\beta}  &\Rightarrow& 1-q^{\alpha{(k-1)}+\beta} \leq  1- q^{\alpha{k}+\beta}\leq 1-q^{2{k}+\beta} \nonumber\\ 
&\overset{ \text{(\ref{additive2.1})}}{\Rightarrow}&
\left(q^{\alpha{(k-1)}+\beta};q\right)_\infty\le\left(q^{\alpha{k}+\beta};q\right)_\infty \le\left(q^{2{k}+\beta} ;q\right)_\infty;
\nonumber\\ 
& \Rightarrow&(1-q)^{\alpha{k}}\leq(1-q)^{\alpha{(k-1)}},
\end{eqnarray} 
and 
\begin{eqnarray}\label{additive3.30}
q^{\beta+1}\leq q\leq q^{1/2}   &\overset{ \text{(\ref{additive2.1})}}{\Rightarrow}&  
\left(q^{1/2};q^{1/2}\right)_k\le\left(q^{\beta+1} ;q^2\right)_k;
\nonumber\\
 (q^\beta;q^2)_k&\overset{ \text{(\ref{additive2.1})}}{=}&
 (1-q^\beta)(1-q^{\beta+2})\dots(1-q^{\beta+k-2}) 
 \nonumber\\&\Rightarrow& (1-q^\beta)^k\leq(q^\beta;q^2)_k.  
\end{eqnarray}

Thus, using (\ref{additive2.2}), (\ref{additive3.29}) and  (\ref{additive3.30}) we find that 
 
\begin{eqnarray}\label{additive3.31}
\gamma_{8,q}&:=&\frac{1}{\Gamma_q(\alpha{k}+\beta)} +  \frac{1}{\Gamma_q(\alpha{(k-1)}+\beta)}\nonumber\\&\overset{ \text{(\ref{additive2.9})}}{=}& 
  (1-q)^{\alpha{k}+\beta-1}\frac{\left(q^{\alpha{k}+\beta};q\right)_\infty}{\left(q;q\right)_\infty}  
  \nonumber\\&+&  (1-q)^{\alpha{(k-1)}+\beta-1}\frac{\left(q^{\alpha{(k-1)}+\beta};q\right)_\infty}{\left(q;q\right)_\infty} \nonumber\\
  &\overset{ \text{(\ref{additive3.29})}}{\leq}&2(1-q)^{\alpha{(k-1)}}\left[(1-q)^{\beta-1}
\frac{\left(q^\beta;q\right)_\infty}{\left(q;q\right)_\infty}  \right]  \frac{\left(q^{2{k}+\beta};q\right)_\infty}{\left(q^\beta;q\right)_\infty} \nonumber\\
  &\overset{ \text{(\ref{additive2.9})(\ref{additive2.2})}}{=}&\frac{2}{\Gamma(\beta)} (1-q)^{\alpha{(k-1)}}
  \frac{1}{\left(q^\beta;q\right)_{2k} } \nonumber\\
  &\overset{ \text{(\ref{additive2.2})}}{=}&\frac{2}{\Gamma(\beta)} 
 (1-q^{1/2})^{\alpha{(k-1)}} \frac{   (1+q^{1/2})^{\alpha{(k-1)}} }{\left(q^\beta;q^2\right)_k\left(q^{\beta+1};q^2\right)_k}\nonumber\\
  &\overset{ \text{(\ref{additive3.30})}}{=}&\frac{1}{\Gamma(\beta)  }   (1-q^{1/2})^{\alpha{(k-1)}}\frac{  2^{\alpha{k}} }{\left(1-q^\beta\right)^k\left(q^{1/2};q^{1/2}\right)_k}\nonumber\\
  &=&\frac{1}{\Gamma(\beta)}(1-q^{1/2})^{\alpha{(k-1)}}\frac{  \left[\frac{2^\alpha}{\left(1-q^\beta\right)(1-q^{1/2})}\right]^k }{ \left[ \frac{(q^{1/2};q^{1/2})_k}{(1-q^{1/2})^k}\right]}\nonumber\\
  &=&\frac{1}{\Gamma(\beta)}(1-q^{1/2})^{\alpha{(k-1)}}\frac{( {\alpha'})^k }
  { [k]_{q^{1/2}}!},
\end{eqnarray}
where $\alpha'=\frac{2^\alpha}{\left(1-q^\beta\right)(1-q^{1/2})}$.

Thus, form  (\ref{additive2.1}), (\ref{additive3.29}) and  (\ref{additive3.31}) we find that
\begin{eqnarray*} 
  \frac{ e_{\alpha,\beta}\left(-z;q\right) }{[1+z]^{-1}}  &\overset{ \text{(\ref{additive1.4})}}{=}&   \sum\limits_{m=0}^\infty \frac{(-z)^k}{\Gamma_q(\alpha{k}+\beta)} +z\sum\limits_{m=0}^\infty \frac{(-z)^k}{\Gamma_q(\alpha{k}+\beta)}  \nonumber\\&\leq&  \frac{1}{\Gamma_q(\beta)}+\sum\limits_{m=1}^\infty \frac{z^k}{\Gamma_q(\alpha{k}+\beta)} +\sum\limits_{m=1}^\infty \frac{z^k}{\Gamma_q(\alpha{(k-1)}+\beta)}\nonumber\\&=& \frac{1}{\Gamma_q(\beta)}+\sum\limits_{m=1}^\infty z^k\left[\frac{1}{\Gamma_q(\alpha{k}+\beta)} +  \frac{1}{\Gamma_q(\alpha{(k-1)}+\beta)} \right]  \nonumber\\&\overset{ \text{(\ref{additive3.31})}}{=}& 
  \frac{1}{\Gamma_q(\beta)}+\sum\limits_{m=1}^\infty z^k\gamma_{8,q}\nonumber\\
 &\overset{ \text{(\ref{additive3.31})}}{\leq}& 
  \frac{1}{\Gamma_q(\beta)}+\frac{1}{\Gamma(\beta) }\sum\limits_{m=1}^\infty    \left[(1-q^{1/2})^{\frac{\alpha{(k-1)}}{k}}z\right]^k\frac{   (\alpha')^k}{ [k]_{q^{1/2}}!}
  \end{eqnarray*}
\begin{eqnarray*}
 &\leq& 
  \frac{1}{\Gamma_q(\beta)}+\frac{1}{\Gamma(\beta) }\sum\limits_{m=1}^\infty  \frac{  (\alpha')^k }  { [k]_{q^{1/2}}!}\nonumber\\
 &\leq& 
  \frac{1}{\Gamma_q(\beta)}\left[1+ \sum\limits_{m=1}^\infty  \frac{  (\alpha')^k }  { [k]_{q^{1/2}}!}\right]
  \nonumber\\
 &\leq& 
  \frac{e_{q^{1/2}}(\alpha')}{\Gamma_q(\beta)}. 
\end{eqnarray*}
Therefore, 
\begin{eqnarray*} 
\left|e_{\alpha,\beta}\left(-z;q\right)\right|&\leq& \frac{e_{q^{1/2}}(\alpha')}{\Gamma_q(\beta)}  \frac{1}{1+z},
\end{eqnarray*}
which  means that the estimate (\ref{additive3.28}) holds.  The proof is complete. 
\end{proof}
 
%%%%%%%%%%%%%%%%%%%%%%%%%%%%%%%%%%%%%%% 
\section{ Direct problem }\label{sc4} 
%%%%%%%%%%%%%%%%%%%%%%%%%%%%%%%%%%%%%%% 

In this section we consider the    Cauchy problem for the time-fractional subdiffusion  in the  quantum calculus. We study existence and uniqueness results for this problem, based on the $\mathcal{L}$–Fourier method. An introduction and some basic definitions of the $\mathcal{L}$-Fourier analysis are given in  \cite{KRT2017}, \cite{RT2018}, \cite{RTT19} and references  therein. We briefly describe  the definitions used in this paper, the global Fourier analysis that has been developed in \cite{RT2016}.  Let  $H$ be a separable Hilbert space and   $\mathcal{L}$ be a positive  operator with the corresponding discrete spectrum  $\{\lambda_k\}_{k\in{I}}$ on $H$, so that  $\lambda_k\in\mathbb{R}$: $\lambda_k\geq 0$ for all $k\in{I}$,  where   $I$ is a countable set $(I=\mathbb{N}^k$ or $I=\mathbb{Z}^k$ for some $k$).   The Plancherel identity on $H$ takes the form
\begin{eqnarray}\label{additive4.1}
 \|u\|_{H}=\left(\sum\limits_{k\in{I}} \left|\langle u,\psi_k\rangle_H\right|^2 \right)^\frac{1}{2},
 \end{eqnarray}
for $u\in{H}$,  where $\psi_k$ is an orthonormal basis of $H$ corresponding to $\lambda_k$.

 Consequently, we can also define the Sobolev spaces $\mathcal{H}^d_{\mathcal{L}}$   associated to $\mathcal{L}$ as  
\begin{eqnarray*}
\mathcal{H}^d_{\mathcal{L}}:=\left\{u\in{H}:  \left(I+\mathcal{L}\right)^{d/2}u\in{H}\right\}
\end{eqnarray*}
for any $d\in\mathbb{R}$. Using Plancherel’s identity, we can write the norm in the following form: 
\begin{eqnarray}\label{additive4.2}
\|u\|_{\mathcal{H}^{d}_{\mathcal{L}}}:=\|\left(I+\mathcal{L}\right)^{d/2}u\|_H=\left(\sum\limits_{k\in{I}}\left(1+\lambda_k\right)^d\left|\langle{u},\psi_k\rangle_H\right|^2\right)^\frac{1}{2}.
\end{eqnarray}

For $u:\in[0,T]\rightarrow\mathcal{H}^d_{\mathcal{L}}$ we introduce the spaces   $W_q^\alpha\left( [0, T]; \mathcal{H}^d_{\mathcal{L}}\right)$,  
$C_q^m\left( [0, T]; \mathcal{H}^d_{\mathcal{L}}\right)$  and $L_q^\infty[0, T]$ as
\begin{eqnarray*}
\|u\|_{W_q^\alpha\left( [0, T];\mathcal{H}^d_{\mathcal{L}}\right)}:= \max\limits_{0\leq t\leq T}\| u(t)\|_{\mathcal{H}^d_{\mathcal{L}}}+
\max\limits_{0\leq t\leq T}\| ^cD^\alpha_{q} u(t)\|_{\mathcal{H}^d_{\mathcal{L}}} 
\end{eqnarray*}
for $0<\alpha<2$, and 
\begin{eqnarray*}
\|u\|_{C_q^m\left( [0, T]; \mathcal{H}^d_{\mathcal{L}}\right)}:=\sum\limits_{i=0}^m\sup\limits_{0<t\leq T}\|D^i_qu(t)\|_{\mathcal{H}^d_{\mathcal{L}}}, 
\end{eqnarray*}
and 
\begin{eqnarray*}
\|u\|_{L_q^\infty\left([0, T];\mathcal{H}^d_{\mathcal{L}}\right)}=\sup\limits_{0\leq{t}\leq{T}}\|u(t)\|_{\mathcal{H}^d_{\mathcal{L}}}.
\end{eqnarray*}

%%%%%%%%%%%%%%%%%%%%%%%%%%%%%%%%%%%%%%%%%%%%%%
 \subsection{The $0 < \alpha \leq 1$ case}
%%%%%%%%%%%%%%%%%%%%%%%%%%%%%%%%%%%%%%%%%%%%%%

The first purpose of this section is to study the direct time-fractional subdiffusion problem for the   equation 
\begin{eqnarray}\label{additive4.3}
^cD^\alpha_{q} u(t)+\mathcal{L}u(t)+mu(t)=f(t) \in H, \;\;t>0,
\end{eqnarray}
 with the initial data
\begin{eqnarray}\label{additive4.4}
u(0) = \varphi\in H,
\end{eqnarray}
where   $^cD^\alpha_{0+,t} u(t)$ is the Caputo fractional $q$-derivatives of order $\alpha$ of $u(t)$ with respect to $t$ (see  (\ref{additive2.10}) and $m>0$.

A generalised solution of Problem (\ref{additive4.3})-(\ref{additive4.4}) is a function 
\begin{eqnarray*}
u\in L_q^\infty\left([0, T]; \mathcal{H}^{d+2}_{\mathcal{L}} \right) {\cap}W_q^\alpha\left([0, T]; \mathcal{H}^d_{\mathcal{L}}\right).
\end{eqnarray*}

For this  considered Problem, the following theorem holds true.

\begin{thm}\label{thm4.1} Let $0<\alpha\leq1$, $0<T<\infty$, $\varphi\in \mathcal{H}^{d+2}_{\mathcal{L}}$ and
$f\in C_q^1\left([0, T]; \mathcal{H}^d_{\mathcal{L}}\right)$. Then there exists a unique solution $u(t)$ of Problem (\ref{additive4.3})-(\ref{additive4.4})  such that \begin{eqnarray*}
u\in L_q^\infty\left([0, T]; \mathcal{H}^{d+2}_{\mathcal{L}} \right) {\cap}W_q^\alpha\left([0, T]; \mathcal{H}^d_{\mathcal{L}}\right). 
\end{eqnarray*}

Moreover, this solution can be written in the form

\begin{eqnarray*} 
u(t)&=&\sum\limits_{k\in{I}}\left[\varphi_k e_{\alpha,1}\left(-\left(\lambda_k+m\right)t^\alpha;q\right)\right]\psi_k\nonumber\\
&+&\sum\limits_{k\in{I}}\left[\int\limits_0^tt^{\alpha-1}\left(qs/t;q\right)_{\alpha-1} \varepsilon^{-q^{\alpha}s}e_{\alpha,\alpha}\left(-\left(\lambda_k+m\right)t^\alpha;q\right) f_k(s)d_qs\right]\psi_k,
\end{eqnarray*}
which satisfies the estimate
 \begin{eqnarray*}
\|^cD^\alpha_{q} u(t)\|^2_{\mathcal{H}^d_{\mathcal{L}}}+\|u(t)\|^2_{\mathcal{H}^{d+2}_{\mathcal{L}}}\leq C_T\left[\|\varphi\|^2_{\mathcal{H}^{d+2}_{\mathcal{L}}} + \|f\|^2_{C_q^1\left([0, T]; \mathcal{H}^d_{\mathcal{L}}\right)}\right], \;\;\;0<t\leq T, 
\end{eqnarray*}
where the $q$-translation operator $\varepsilon^{-q^{\alpha}s}$ is defined in (\ref{additive2.6}) and $C_T:=\max\{2, T\}$.
\end{thm}

\begin{proof} {\it Existence}.  Since the system of eigenfunctions $\psi_k$ is a basis in $H$, we seek a generalised solution by $u(t)$ in the form
\begin{eqnarray}\label{additive4.5}
u(t)=\sum\limits_{k\in{I}}u_k(t)\psi_k,
\end{eqnarray}
and we represent the function  on the right-hand side of the equation (\ref{additive4.2}) in  the form: 
\begin{eqnarray}\label{additive4.6}
f(t)=\sum\limits_{k\in{I}}f_k(t)\psi_k, 
\end{eqnarray}
where the Fourier coefficients are defined by the formula
\begin{eqnarray*} 
u_k(t):=\langle u(t), \psi_k\rangle_H,\;\;\;\;f_k(t):=\langle f(t), \psi_k\rangle_H.
\end{eqnarray*}

Moreover,  we find the expression for the derivatives of (\ref{additive4.5}) -(\ref{additive4.6}):
\begin{eqnarray}\label{additive4.7}
^cD^\alpha_{q} u(t)=\sum\limits_{k\in{I}} {^cD}^\alpha_{q}u_k(t)\psi_k,
\end{eqnarray}
\begin{eqnarray}\label{additive4.8}
\mathcal{L}u(t)=\sum\limits_{k\in{I}}\lambda_ku_k(t)\psi_k. 
\end{eqnarray}

 Then substituting (\ref{additive4.6}) -(\ref{additive4.8})  into (\ref{additive4.3}) and (\ref{additive4.4}), we get
\begin{eqnarray}\label{additive4.9}
 ^cD^\alpha_{q,t} u_k(t)+\left(\lambda_k+m\right)u_k(t)=f_k(t), 
\end{eqnarray}
\begin{eqnarray}\label{additive4.10}
 u_k(0)= \varphi_k,
 \end{eqnarray}
for $k\in{I}$ and $t>0$.

A generalised solution of the problem  (\ref{additive4.8}) - (\ref{additive4.10})  (\cite[
Theorem 3.1]{STT2022}) 
  is given by
\begin{eqnarray}\label{additive4.11}
 u_k(t)&=& \varphi_k e_{\alpha,1}\left(-\left(\lambda_k+m\right)t^\alpha;q\right)\nonumber\\
&+&\int\limits_0^tt^{\alpha-1}\left(qs/t;q\right)_{\alpha-1} \varepsilon^{-q^{\alpha}s}e_{\alpha,\alpha}\left(-\left(\lambda_k+m\right)t^\alpha;q\right) f_k(s)d_qs. 
\end{eqnarray}

This, combined with (\ref{additive4.5})  and (\ref{additive4.11})  gives
\begin{eqnarray*}
u(t)&=&\sum\limits_{k\in{I}}\left[\varphi_k e_{\alpha,1}\left(-\left(\lambda_k+m\right)t^\alpha;q\right)\right]\psi_k\nonumber\\
&+&\sum\limits_{k\in{I}}\left[\int\limits_0^tt^{\alpha-1}\left(qs/t;q\right)_{\alpha-1} \varepsilon^{-q^{\alpha}s}e_{\alpha,\alpha}\left(-\left(\lambda_k+m\right)t^\alpha;q\right) f_k(s)d_qs\right]\psi_k.
\end{eqnarray*}

{\it Convergence}. We assume that     $i=0,1,2\dots,$ and $0<\alpha\leq1$. Then by following (see \cite[Proposition 14.1]{ChKa2000}) we  have   
\begin{eqnarray}\label{additive4.12}
  t^{\alpha{i}-1} (q^{\alpha}s/t;q)_{\alpha i-1}= t^{\alpha(i-1)}(q^\alpha s;q)_{\alpha(i-1)} \left[t^{\alpha-1}(qs/t;q)_{\alpha-1}\right].  
\end{eqnarray}

For a fixed $k$, let us denote  $\lambda_0:=\lambda_k+m$ for $m>0$. Then by making use of   (\ref{additive2.5}), 
 (\ref{additive2.6}), (\ref{additive2.9})  and (\ref{additive4.12}),    we have 

\begin{eqnarray*} 
D_{q,s}\left[\varepsilon^{-q^{\alpha}s}e_{\alpha,1}\left(-\lambda_0t^\alpha;q\right)\right]&=&D_{q,s}\left[e_{\alpha,1}\left(-\lambda_0t^\alpha( q^{\alpha}s/t;q)_\alpha;q\right)\right]\nonumber\\
&\overset{ \text{(\ref{additive2.6})}}{=}& 
\sum\limits_{i=1}^\infty \frac{\left(-\lambda_0\right)^i }{\Gamma_q(\alpha{i}+1)}D_{q,s}\left[t^{\alpha{i}}(q^{\alpha}s/t;q)_{\alpha i}\right]\nonumber\\
&\overset{ \text{(\ref{additive2.5})}}{=}&-\sum\limits_{i=1}^\infty \left(-\lambda_0\right)^i \frac{[\alpha{i}]_q}{\Gamma_q(\alpha{i}+1)}t^{\alpha{i}-1} (q^{\alpha}s/t;q)_{\alpha i-1} \nonumber\\
&\overset{ \text{(\ref{additive2.9})}}{=}&-\sum\limits_{i=1}^\infty  \frac{\left(-\lambda_0\right)^i}{\Gamma_q(\alpha{i})}t^{\alpha{i}-1} (q^{\alpha}s/t;q)_{\alpha i-1} \nonumber\\
&\overset{ \text{(\ref{additive4.12})}}{=}&\lambda_0\sum\limits_{i=1}^\infty \frac{\left(-\lambda_0\right)^{i-1}}{\Gamma(\alpha i)} \left[\frac{t^{\alpha(i-1)}(q^\alpha s;q)_{\alpha(i-1)}}{\left[t^{\alpha-1}(qs/t;q)_{\alpha-1}\right]^{-1}}\right]\nonumber\\
&=&\lambda_0\frac{\sum\limits_{i=0}^\infty \frac{\left(-\lambda_k\right)^i t^{\alpha{i}}(q^\alpha{s}/t;q)_{\alpha{i}}}{\Gamma(\alpha i+\alpha)}}{\left[t^{\alpha-1}(qs/t;q)_{\alpha-1}\right]^{-1}} \nonumber\\
&=&\lambda_0\frac{\varepsilon^{-q^{\alpha}s}e_{\alpha,\alpha}\left(-\lambda_0t^\alpha;q\right)}{\left[t^{\alpha-1}(qs/t;q)_{\alpha-1}\right]^{-1}}.
\end{eqnarray*}

Thus,  
\begin{eqnarray}\label{additive4.13}
\frac{\varepsilon^{-q^{\alpha}s}e_{\alpha,\alpha}\left(-\lambda_0t^\alpha;q\right)}{\left[t^{\alpha-1}(qs/t;q)_{\alpha-1}\right]^{-1}}=\frac{D_{q,s}\left[\varepsilon^{-q^{\alpha}s}e_{\alpha,1}\left(-\lambda_0t^\alpha;q\right)\right]}{\lambda_0}. 
\end{eqnarray}

Form (\ref{additive2.8}) and (\ref{additive4.13})  we get that 
\begin{eqnarray}\label{additive4.14}
\int\limits_0^t \frac{\varepsilon^{-q^{\alpha}s}e_{\alpha,\alpha}\left(-\lambda_0t^\alpha;q\right)
 f_k(s)}{\left[t^{\alpha-1}\left(qs/t;q\right)^{\alpha-1}\right]^{-1}}d_qs
&\overset{ \text{(\ref{additive4.13})}}{=}&\int\limits_0^t\frac{D_{q,s}\left[\varepsilon^{-q^{\alpha}s}e_{\alpha,1}\left(-\lambda_0 t^\alpha ;q\right)\right] f_k(s)d_qs}{\lambda_0}\nonumber\\
&\overset{ \text{(\ref{additive2.8})}}{=}&\frac{f_k(t)\varepsilon^{-q^{\alpha}t}e_{\alpha,1}\left(-\lambda_0 t^\alpha ;q\right)}{\lambda_0}
-\frac{f_k(0)e_{\alpha,1}\left(-\lambda_0t^\alpha;q\right)}{\lambda_0}\nonumber\\
&-&\int\limits_0^t\frac{\varepsilon^{-q^{\alpha}s}e_{\alpha,1}\left(-\lambda_0t^\alpha;q\right)D_{q}f_k(s)d_qs}{\lambda_0}. 
\end{eqnarray}

Since $\varphi\in  \mathcal{H}^{d+2}_{\mathcal{L}}$, $f\in C_q^1\left([0, T]; \mathcal{H}^d_{\mathcal{L}}\right)$, then from (\ref{additive3.28}), (\ref{additive4.11}) and (\ref{additive4.14})   it follows that
\begin{eqnarray*} 
\left|\langle u(t), \psi_k\rangle_H\right|&\overset{ \text{(\ref{additive4.11})}}{\leq}& \left| \langle \varphi(t), \psi_k\rangle_H\right| \left|e_{\alpha,1}\left(-\lambda_0t^\alpha;q\right)\right|\nonumber\\
&+&\left|\int\limits_0^tt^{\alpha-1}\left(qs/t;q\right)_{\alpha-1} \varepsilon^{-q^{\alpha}s}e_{\alpha,\alpha}\left(-\lambda_0t^\alpha;q\right) \langle f(t), \psi_k\rangle_Hd_qs\right|\nonumber\\
&\overset{ \text{(\ref{additive4.14})}}{\leq}&\left|\langle \varphi(t), \psi_k\rangle_H\right|\left|e_{\alpha,1}\left(-\lambda_0t^\alpha;q\right)\right|+
\frac{\left|\langle f(t), \psi_k\rangle_H\right|\left|\varepsilon^{-q^{\alpha}t}e_{\alpha,1}\left(-\lambda_0 t^\alpha ;q\right)\right|}{\lambda_0}\nonumber\\
&+&\frac{\left| e_{\alpha,1}\left(-\lambda_0t^\alpha;q\right)\right|\left|\langle f(0), \psi_k\rangle_H\right|}{\lambda_0}
\end{eqnarray*}
\begin{eqnarray*} 
&+&\frac{ \int\limits_0^t\left|\varepsilon^{-q^{\alpha}s}e_{\alpha,1}\left(-\lambda_0t^\alpha;q\right)\langle D_{q}f(t), \psi_k\rangle_H\right|d_qs }{\lambda_0}\nonumber\\
&\overset{ \text{(\ref{additive3.28})}}{\lesssim}&\left|\langle \varphi(t), \psi_k\rangle_H\right|+ \frac{\left|\langle f(t), \psi_k\rangle_H\right|}{\lambda_0} \nonumber\\
&+&\frac{\left|\langle f(0), \psi_k\rangle_H\right|}{\lambda_0}
+\frac{ 1}{\lambda_0}\int\limits_0^t \left|D_{q}\langle f(t), \psi_k\rangle_H\right|d_qs\nonumber\\
&\leq&\left|\langle\varphi, \psi_k\rangle_H\right| +\frac{2}{\lambda_0}\sup\limits_{0<t\leq{T}}\left|\langle f(t), \psi_k\rangle_H\right|
+\frac{ T}{\lambda_0}\sup\limits_{0<t\leq{T}}\left|\langle D_{q}f(t), \psi_k\rangle_H\right|,
\end{eqnarray*}
which implies 
\begin{eqnarray}\label{additive4.15}
\left(1+\lambda_k\right)^{d/2}\left|\langle u(t), \psi_k\rangle_H\right|&\leq& C_T\left[\left|\langle\left(1+\lambda_k\right)^{d/2}\varphi, \psi_k\rangle_H\right|\right.\nonumber\\
&+& \sup\limits_{0<t\leq{T}}\left|\langle \left(1+\lambda_k\right)^{d/2-1}f(t), \psi_k\rangle_H\right|\nonumber\\
&+& \left.\sup\limits_{0<t\leq{T}}\left|\langle \left(1+\lambda_k\right)^{d/2-1}D_{q}f(t), \psi_k\rangle_H\right|\right], 
\end{eqnarray}
where $C_T:=\max\{2, T\}$.

By the definition of the Fourier coefficient  and   (\ref{additive4.9}) and (\ref{additive4.15}), we find that    
\begin{eqnarray}\label{additive4.16}
\left(1+\lambda_k\right)^{d/2}\left|\langle {^cD^\alpha_{q}}u(t), \psi_k\rangle_H\right|&=&\left(1+\lambda_k\right)^{d/2}|^cD^\alpha_{q} u_k(t)|\nonumber\\
&\overset{ \text{(\ref{additive4.9})}}{\lesssim}&\left(1+\lambda_k\right)^{d/2+1}\left|u_k(t)\right|+\left(1+\lambda_k\right)^{d/2}\left|f_k(t)\right|\nonumber\\
&\overset{ \text{(\ref{additive4.15})}}{\leq}& C_T\left[\left|\langle\left(1+\lambda_k\right)^{d/2+1}\varphi, \psi_k\rangle_H\right|\right.\nonumber\\
&+& \sup\limits_{0<t\leq{T}}\left|\langle \left(1+\lambda_k\right)^{d/2}f(t), \psi_k\rangle_H\right| \nonumber\\
&+& \left.\sup\limits_{0<t\leq{T}}\left|\langle \left(1+\lambda_k\right)^{d/2}D_{q}f(t), \psi_k\rangle_H\right|\right]\nonumber\\
&=&C_T\left[\left|\langle\left(I+\mathcal{L}\right)^{d/2+1}\varphi, \psi_k\rangle_H\right|\right.\nonumber\\
&+& \sup\limits_{0<t\leq{T}}\left|\langle \left(I+\mathcal{L}\right)^{d/2}f(t), \psi_k\rangle_H\right| \nonumber\\
&+& \left.\sup\limits_{0<t\leq{T}}\left|\langle\left(I+\mathcal{L}\right)^{d/2}D_{q}f(t), \psi_k\rangle_H\right|\right], 
\end{eqnarray}
and 
\begin{eqnarray}\label{additive4.17}
\left(1+\lambda_k\right)^{d/2}\left|\langle\mathcal{L}u(t), \psi_k\rangle_H\right|&=&\left(1+\lambda_k\right)^{d/2}\left|\langle\lambda_ku(t), \psi_k\rangle_H\right|\nonumber\\
&\overset{ \text{(\ref{additive4.15})}}{\leq}&C_T\left[\left|\langle\left(1+\lambda_k\right)^{d/2+1}\varphi, \psi_k\rangle_H\right|\right.\nonumber\\
&+& \sup\limits_{0<t\leq{T}}\left|\langle \left(1+\lambda_k\right)^{d/2}f(t), \psi_k\rangle_H\right|\nonumber\\
&+&\left.\sup\limits_{0<t\leq{T}}\left|\langle \left(1+\lambda_k\right)^{d/2}D_{q}f(t), \psi_k\rangle_H\right|\right]\nonumber\\
&=&C_T\left[\left|\langle\left(I+\mathcal{L}\right)^{d/2+1}\varphi, \psi_k\rangle_H\right| \right.\nonumber\\
&+& \sup\limits_{0<t\leq{T}}\left|\langle \left(I+\mathcal{L}\right)^{d/2}f(t), \psi_k\rangle_H\right| \nonumber\\
&+& \left.\sup\limits_{0<t\leq{T}}\left|\langle \left(I+\mathcal{L}\right)^{d/2}D_{q}f(t), \psi_k\rangle_H\right|\right].
\end{eqnarray}

Using (\ref{additive4.2}),  (\ref{additive4.16}), (\ref{additive4.17}),  and the Parseval  identity, we get that
\begin{eqnarray*}
\|^cD^\alpha_{q} u(t)\|^2_{\mathcal{H}^d_{\mathcal{L}}}&\overset{ \text{(\ref{additive4.2})}}{=}&\sum\limits_{k\in I} \left(1+\lambda_k\right)^{d}\left|\langle {^cD^\alpha_{q}}u(t), \psi_k\rangle_H\right|^2\nonumber\\
&\overset{ \text{(\ref{additive4.16})}}{\leq}& 
 C_T\left\{\sum\limits_{k\in I}\left|\langle\left(I+\mathcal{L}\right)^{d/2+1}\varphi, \psi_k\rangle_H\right|^2\right.\nonumber\\
 &+& \sum\limits_{k\in I}\left[\sup\limits_{0<t\leq{T}}\left|\langle \left(I+\mathcal{L}\right)^{d/2}f(t), \psi_k\rangle_H\right|^2\right. \nonumber\\
&+&\left. \left.\sup\limits_{0<t\leq{T}}\left|\langle \left(I+\mathcal{L}\right)^{d/2}D_{q}f(t), \psi_k\rangle_H\right|^2\right]\right\}\nonumber\\ 
&\overset{ \text{(\ref{additive4.2})}}{=}&C_T\left[\|\left(I+\mathcal{L}\right)^{d/2+1}\varphi\|^2_H+\sup\limits_{0<t\leq{T}}\|\left(I+\mathcal{L}\right)^{d/2}f(t)\|^2_H\right.\nonumber\\
&+&\left.\sup\limits_{0<t\leq{T}}\|\left(1+\mathcal{L}\right)^{d/2}D_{q}f(t)\|^2_H\right]\nonumber\\ 
&=&C_T\left[\| \varphi\|^2_{\mathcal{H}^{d+2}_{\mathcal{L}}} + \|f\|^2_{C_q^1\left([0, T]; \mathcal{H}^d_{\mathcal{L}}\right)}\right]<\infty,
\end{eqnarray*}
and 
\begin{eqnarray*}
\|\left(I+\mathcal{L}\right)u(t)\|^2_{\mathcal{H}^d_{\mathcal{L}}}&\overset{ \text{(\ref{additive4.2})}}{=}&\sum\limits_{k\in I}\left|\left(1+\lambda_k\right)^{d/2}\langle\left(I+\mathcal{L}\right)u(t), \psi_k\rangle_H\right|^2\\
&\overset{ \text{(\ref{additive4.17})}}{\leq}&C_T\left[\|\varphi\|^2_{\mathcal{H}^{d+2}_{\mathcal{L}}} + \|f\|^2_{C_q^1\left([0, T]; \mathcal{H}^d_{\mathcal{L}}\right)}\right]<\infty.
\end{eqnarray*}

Hence, the above estimates imply that
\begin{eqnarray*}
\|^cD^\alpha_{q} u(t)\|^2_{\mathcal{H}^d_{\mathcal{L}}}+\|u(t)\|^2_{\mathcal{H}^{d+2}_{\mathcal{L}}}\leq C_T\left[\|\varphi\|^2_{\mathcal{H}^{d+2}_{\mathcal{L}}} + \|f\|^2_{C_q^1\left([0, T]; \mathcal{H}^d_{\mathcal{L}}\right)}\right], 
\end{eqnarray*}
which means that
$ u\in L_q^\infty\left([0, T]; \mathcal{H}^{d+2}_{\mathcal{L}} \right) {\cap}W_q^\alpha\left([0, T]; \mathcal{H}^d_{\mathcal{L}}\right)$.

{\it The uniqueness}. We assume that there are two solutions $v_1(t)$ and $v_2(t)$  of    the problem (\ref{additive4.2})-(\ref{additive4.3}). Then we must have that, for $w=u-v$,
\begin{eqnarray*}
^cD^\alpha_{q}w+\mathcal{L}w&=& ^cD^\alpha_{q}\left(u-v\right)+\mathcal{L}\left(u-v\right)+mw\left(u-v\right)\\
&=&^cD^\alpha_{t}u+\mathcal{L}u-\left[^cD^\alpha_{q}v+\mathcal{L}v\right]\\
&=&f-f=0,
\end{eqnarray*}
with  the initial conditions
$$
w(0)=u(0)-v(0)=\varphi-\varphi=0.
$$

Therefore, $w=0$, and the solution must be unique. The proof is complete.
\end{proof}

\;\;\;\;\;

%%%%%%%%%%%%%%%%%%%%%%%%%%%
\subsection{The $1< \alpha < 2$ case}
%%%%%%%%%%%%%%%%%%%%%%%%%%
This subsection is concerned with a   Cauchy  problem for   the  time-fractional subdiffusion equation 
\begin{eqnarray}\label{additive4.18}
^cD^\alpha_{q} u(t)+\mathcal{L}u(t)+mu(t)=f(t) \in H,
\end{eqnarray}
with the initial data 
\begin{eqnarray}\label{additive4.19}
u(0) = \varphi\in{H},  \;\;\;\;   D_qu(0) = \rho\in{H}.    
\end{eqnarray}
 
\begin{thm}\label{thm4.2} Let  $1<\alpha<2$, $m>0$ and $0<T<\infty$. We assume that  $\varphi, \rho\in  \mathcal{H}^{d+2}_{\mathcal{L}}$ and  $f\in C_q^1\left([0, T]; \mathcal{H}^d_{\mathcal{L}}\right)$. Then there exists a unique solution of Problem (\ref{additive4.18})-(\ref{additive4.19}):
\begin{eqnarray*}
u\in L_q^\infty\left([0, T]; \mathcal{H}^{d+2}_{\mathcal{L}}\right){\cap}  W_q^\alpha\left([0, T]; \mathcal{H}^d_{\mathcal{L}}\right). 
\end{eqnarray*}

Moreover, this solution can be written in the form
\begin{eqnarray}\label{additive4.20}
 u(t) &=&\sum\limits_{k\in I}\left[\varphi_ke_{\alpha,1}\left(-\left(\lambda_k+m\right)t^\alpha;q\right)+t\rho_ke_{\alpha,2}\left(-\left(\lambda_k+m\right)t^\alpha;q\right)\right.\nonumber\\
  &+&\left.\int\limits_0^tt^{\alpha-1}\left(qs/t;q\right)_{\alpha-1}\varepsilon^{-q^{\alpha}s}e_{\alpha,\alpha}\left(-\left(\lambda_k+m\right)t^\alpha;q\right)f_k(s)d_qs\right]\phi_k,  
\end{eqnarray}
which satisfies the estimate
\begin{equation}\label{additive4.21}
\|^cD^\alpha_{q} u(t)\|^2_{\mathcal{H}^d_{\mathcal{L}}}+\|u(t)\|^2_{\mathcal{H}^{d+2}_{\mathcal{L}}}\leq  C_T\left[\|\varphi\|^2_{\mathcal{H}^{d+2}_{\mathcal{L}}}+\|\rho\|^2_{\mathcal{H}^{d+2}_{\mathcal{L}}} + \|f\|^2_{C_q^1\left([0, T]; \mathcal{H}^d_{\mathcal{L}}\right)} \right],  
\end{equation}
for $0\leq{t}\leq{T}$.
\end{thm}

\begin{proof}
{\it Existence.}  By repeating the arguments in the proof of Theorem \ref{thm4.1}, we   have the equation (\ref{additive4.7}) with  the initial conditions
\begin{eqnarray}\label{additive4.22}
 u_k(0) =\varphi_k, \;\;\;  D_{q}u_k(0) =\rho_k,\;\;\;k\in{I},  
\end{eqnarray}
and  a generalised solution in the following form: 
\begin{eqnarray}\label{additive4.23}
 u_k(t) &=&\varphi_ke_{\alpha,1}\left(-\left(\lambda_k+m\right)t^\alpha;q\right)+t\rho_ke_{\alpha,2}\left(-\left(\lambda_k+m\right)t^\alpha;q\right) \nonumber\\
 &+&\int\limits_0^tt^{\alpha-1}\left(qs/t;q\right)_{\alpha-1}\varepsilon^{-q^{\alpha}s}e_{\alpha,\alpha}\left(-\left(\lambda_k+m\right)t^\alpha;q\right)f_k(s)d_qs.
\end{eqnarray}

Using (\ref{additive4.5})  and  (\ref{additive4.23})  we obtain  the solution  (\ref{additive4.20}).

{\it Convergence}.    By using (\ref{additive4.14})  and  (\ref{additive4.23}), we get that 
\begin{eqnarray}\label{additive4.24}
 u_k(t) &\overset{ \text{(\ref{additive4.14})(\ref{additive4.23})}}{=}& \varphi_ke_{\alpha,1}\left(-\left(\lambda_k+m\right)t^\alpha;q\right)+t\rho_ke_{\alpha, 2}\left(-\left(\lambda_k+m\right)t^\alpha;q\right) \nonumber\\
 &+&\frac{1}{\lambda_k+m}f_k(t)e_{\alpha,1}\left(-\left(\lambda_k+m\right)t^\alpha(q^\alpha;q)_\alpha;q\right)\nonumber\\
 &-&\frac{1}{\lambda_k+m}f_k(0)e_{\alpha,1}\left(-\left(\lambda_k+m\right)t^\alpha;q\right)\nonumber\\
 &-&\frac{1}{\lambda_k+m}\int\limits_0^t \varepsilon^{-q^{\alpha+1}s}e_{\alpha,1}\left(-\left(\lambda_k+m\right)t^\alpha;q\right) D_{q,s}f_k(s)d_qs.
\end{eqnarray}

From (\ref{additive3.28}) and (\ref{additive4.24})   we conclude that
\begin{eqnarray}\label{additive4.25}
\left(1+\lambda_k \right)^{d/2}\left|\langle{u}(t),\psi_k\rangle_H \right|&\overset{ \text{(\ref{additive3.28})(\ref{additive4.24})}}{\leq}& C_T\left\{ \left(1+\lambda_k\right)^{d/2}\left|\langle\varphi,  \psi_k\rangle_H\right| \right.\nonumber\\
&+&\left(1+\lambda_k\right)^{d/2}\left|\langle\rho,  \psi_k\rangle_H \right| \nonumber\\
&+&\left(1+\lambda_k\right)^{d/2-1} \sup\limits_{0<{s}\leq{T}}\left|\langle f(t,\cdot),  \psi_k\rangle_H\right| \nonumber\\
&+&\left. \left(1+\lambda_k\right)^{d/2-1} \sup\limits_{0<{s}\leq{T}}\left|\langle D_q f(s,\cdot),  \psi_k\rangle_H\right|\right\}, 
\end{eqnarray}
where $C_T:=\max\{2, T\}$.

Since $\varphi, \rho\in  \mathcal{H}^{d+2}_{\mathcal{L}}$, $f\in C_q^1\left([0, T]; \mathcal{H}^d_{\mathcal{L}}\right)$,  using (\ref{additive4.9})  and (\ref{additive4.25})   we have that  \begin{eqnarray}\label{additive4.26}
\left(1+\lambda_k\right)^{d/2}\left|\langle ^cD^\alpha_{q}u(t), \psi_k\rangle_H\right|&\overset{ \text{(\ref{additive4.9})}}{\lesssim}&\left(1+\lambda_k\right)^{d/2}\left|u_k(t)\right|+\left|f_k(t)\right|\nonumber\\
&\overset{ \text{(\ref{additive4.25})}}{\leq}& 
 C_T\left\{\left|\langle\left(I+\mathcal{L}\right)^{d/2+1}\varphi, \psi_k\rangle_H\right|\right.\nonumber\\
 &+& \left|\langle\left(I+\mathcal{L}\right)^{d/2+1}\rho, \psi_k\rangle_H\right| \nonumber\\
&+&  \sup\limits_{0<{s}\leq{T}}\left|\langle \left(I+\mathcal{L}\right)^{d/2}f(t), \psi_k\rangle_H\right| \nonumber\\
&+&\left. \sup\limits_{0<{s}\leq{T}}\left|\langle \left(I+\mathcal{L}\right)^{d/2}D_{q}f(t), \psi_k\rangle_H\right|\right\},
\end{eqnarray} 
 and 
\begin{eqnarray}\label{additive4.27}
\left(1+\lambda_k\right)^{d/2}\left|\langle\left(I+\mathcal{L}\right)u(t), \psi_k\rangle_H\right| 
&\overset{ \text{\ref{additive4.25})}}{\leq}&C_T\left\{\left|\langle\left(I+\mathcal{L}\right)^{d/2+1}\varphi, \psi_k\rangle_H\right|\right.\nonumber\\
&+&\left|\langle\left(I+\mathcal{L}\right)^{d/2+1}\rho, \psi_k\rangle_H\right|  \nonumber\\
&+&  \sup\limits_{0<{s}\leq{T}}\left|\langle \left(I+\mathcal{L}\right)^{d/2}f(t), \psi_k\rangle_H\right|\nonumber\\
&+&\left.\sup\limits_{0<{s}\leq{T}}\left|\langle \left(I+\mathcal{L}\right)^{d/2}D_{q}f(t), \psi_k\rangle_H\right|\right\}.
\end{eqnarray}

According  (\ref{additive4.26})  and  (\ref{additive4.27})     we get  (\ref{additive4.21}).

Hence, the above estimates imply that
\begin{eqnarray*}
\|^cD^\alpha_{q} u(t)\|^2_{\mathcal{H}^d_{\mathcal{L}}}+\|u(t)\|^2_{\mathcal{H}^{d+2}_{\mathcal{L}}}\leq C_T\left[\|\varphi\|^2_{\mathcal{H}^{d+2}_{\mathcal{L}}} +\|\rho\|^2_{\mathcal{H}^{d+2}_{\mathcal{L}}}+ \|f\|^2_{C_q^1\left([0, T]; \mathcal{H}^d_{\mathcal{L}}\right)}\right], 
\end{eqnarray*}
which means that
$ u\in L_q^\infty\left([0, T]; \mathcal{H}^{d+2}_{\mathcal{L}} \right) \cap W_q^\alpha\left([0, T]; \mathcal{H}^d_{\mathcal{L}}\right)$.

{\it Uniqueness.} This part can be proved  similarly to the proof of  Theorem \ref{thm4.1},   so we omit the details.  The proof is complete.
\end{proof}

%%%%%%%%%%%%%%%%%%%%%%%%%%%%%
\section{Inverse problem}\label{sc5} 
%%%%%%%%%%%%%%%%%%%%%%%%%%%%%%%%

The subsection deals with an inverse problem concerning the time-fractional diffusion
equation. The problem is to find the couple  $(u(t), f )$ satisfying the equation:
\begin{eqnarray}\label{additive5.1}
^cD^\alpha_{q,t} u(t)+\mathcal{L}u(t)+mu(t)=f\in{H},   \;\;\;  t>0,
\end{eqnarray}
with initial data
\begin{eqnarray}\label{additive5.2}
u(0)&=&\varphi\in{H},
\end{eqnarray}
and final condition:
\begin{eqnarray}\label{additive5.3}
u(T)&=&\rho\in{H}, 
\end{eqnarray}
 where $m>0$.

\begin{thm}\label{thm5.1} Let $0<\alpha<1$ and $0<T<\infty$.  Assume that $\varphi, \rho\in \mathcal{H}^{d+2}_{\mathcal{L}}$. Then the generalised solution of  (\ref{additive5.1})-(\ref{additive5.3}), exists, is unique, and can be written in the form
\begin{eqnarray*} 
u(t)&=&\varphi  
+\sum\limits_{k\in{I}}\frac{ \left[\varphi_k- \rho_k\right]\left[e_{\alpha,1}\left(-\left(\lambda_k+m\right)t^\alpha;q\right)-1\right]\psi_k}{ \left[1-e_{\alpha,1}\left(-\left(\lambda_k+m\right)T^\alpha;q\right)\right]}, 
\end{eqnarray*}
and 
\begin{eqnarray*} 
f&=& \mathcal{L}\varphi-\sum\limits_{k\in{I}}\frac{\left(\lambda_k+m\right)\left[\varphi_k- \rho_k\right]\psi_k}{1-e_{\alpha,1}\left(-\left(\lambda_k+m\right)T^\alpha;q\right)}.
\end{eqnarray*}

Moreover,  
\begin{eqnarray*} 
u\in L^\infty_q\left([0, T]; \mathcal{H}^{d+2}_{\mathcal{L}}\right)\cap W^\alpha_q\left([0, T]; \mathcal{H}^d_{\mathcal{L}}\right);\;\;\; f\in\mathcal{H}^d_{\mathcal{L}},
\end{eqnarray*}
which satisfy  the estimates
\begin{eqnarray}\label{additive5.4}
\|^cD^\alpha_{q} u(t)\|^2_{\mathcal{H}^d_{\mathcal{L}}}&+&\|u(t)\|^2_{\mathcal{H}^{d+2}_{\mathcal{L}}}\lesssim  \| \varphi\|^2_{\mathcal{H}^{d+2}_{\mathcal{L}}}+\| \rho\|^2_{\mathcal{H}^{d+2}_{\mathcal{L}}};    \nonumber\\
\|f\|^2_{\mathcal{H}^d_{\mathcal{L}}}&\lesssim&\| \varphi\|^2_{\mathcal{H}^{d+2}_{\mathcal{L}}}+\| \rho\|^2_{\mathcal{H}^{d+2}_{\mathcal{L}}}.  
\end{eqnarray}
\end{thm}
 \begin{proof} {\it Existence.}  We partly repeat the arguments in the proof of Theorem \ref{thm4.1}. For this, substituting (\ref{additive4.5}) -(\ref{additive4.7})  into (\ref{additive5.1}), we have  
\begin{eqnarray}\label{additive5.5}
^cD^\alpha_{q,t} u_k(t)+\left(\lambda_k+m\right)u_k(t)=f_k,   
\end{eqnarray}
with   initial data
\begin{eqnarray}\label{additive5.6}
u_k(0)&=& \varphi_k, 
\end{eqnarray}
and the final condition  
\begin{eqnarray}\label{additive5.7}
u_k(T)&=&\rho_k,  
\end{eqnarray}
for all $k\in{I}$.

By using (\ref{additive4.11}) and (\ref{additive4.13}) we find  a general solution of the problem (\ref{additive5.5})-(\ref{additive5.7}): 
\begin{eqnarray}\label{additive5.8}
 u_k(t)&\overset{ \text{(\ref{additive4.11})}}{=}& \varphi_k e_{\alpha,1}\left(-\left(\lambda_k+m\right)t^\alpha;q\right)\nonumber\\
&+&f_k\int\limits_0^tt^{\alpha-1}\left(qs/t;q\right)_{\alpha-1} \varepsilon^{-q^{\alpha}s}e_{\alpha,\alpha}\left(-\left(\lambda_k+m\right)t^\alpha;q\right) d_qs\nonumber\\
&\overset{ \text{(\ref{additive4.13})}}{=}& \varphi_k e_{\alpha,1}\left(-\left(\lambda_k+m\right)t^\alpha;q\right)\nonumber\\
&+&\frac{f_k}{\lambda_k+m}\int\limits_0^tD_{q,s}\left[\varepsilon^{-s}e_{\alpha,1}\left(-\left(\lambda_k+m\right)t^\alpha;q\right)\right] d_qs\nonumber\\
&=& \varphi_k e_{\alpha,1}\left(-\left(\lambda_k+m\right)t^\alpha;q\right)+\frac{f_k}{\lambda_k+m}\left[1-e_{\alpha,1}\left(-\left(\lambda_k+m\right)t^\alpha;q\right)\right],  
\end{eqnarray}
where the constant    $f_k$ needs to be determined.

Now, by using the conditions   (\ref{additive5.6}), (\ref{additive5.7}) and (\ref{additive5.8})   we get that 
\begin{eqnarray*}
u_k(T)\overset{ \text{(\ref{additive5.7})}}{=}\rho_k&\overset{ \text{(\ref{additive5.6})(\ref{additive5.8})}}{=}& \varphi_k e_{\alpha,1}\left(-\left(\lambda_k+m\right)T^\alpha;q\right)\nonumber\\
&+&\frac{f_k}{\lambda_k+m}\left[1-e_{\alpha,1}\left(-\left(\lambda_k+m\right)T^\alpha;q\right)\right],
\end{eqnarray*}
so that 
\begin{eqnarray}\label{additive5.9}
f_k&=&\frac{\left(\lambda_k+m\right)\rho_k-\left(\lambda_k+m\right)\varphi_k e_{\alpha,1}\left(-\left(\lambda_k+m\right)T^\alpha;q\right)}{1-e_{\alpha,1}\left(-\left(\lambda_k+m\right)T^\alpha;q\right)}\nonumber\\
&=& \left(\lambda_k+m\right)\varphi_k-\frac{\left(\lambda_k+m\right)\varphi_k-\left(\lambda_k+m\right)\rho_k}{1-e_{\alpha,1}\left(-\left(\lambda_k+m\right)T^\alpha;q\right)}.
\end{eqnarray}

From (\ref{additive5.8})-(\ref{additive5.9}) it is follows that 
\begin{eqnarray}\label{additive5.10}
u_k&\overset{ \text{(\ref{additive5.8})(\ref{additive5.9})}}{=}&\varphi_k e_{\alpha,1}\left(-\left(\lambda_k+m\right)t^\alpha;q\right)+\varphi_k\left[1-e_{\alpha,1}\left(-\left(\lambda_k+m\right)t^\alpha;q\right)\right]\nonumber\\
&-&\frac{\left[\left(\lambda_k+m\right)\varphi_k-\left(\lambda_k+m\right)\rho_k\right]\left[1-e_{\alpha,1}\left(-\left(\lambda_k+m\right)t^\alpha;q\right)\right]}{\left(\lambda_k+m\right)\left[1-e_{\alpha,1}\left(-\left(\lambda_k+m\right)T^\alpha;q\right)\right]}\nonumber\\
&=&\varphi_k  
+\frac{\left[ \varphi_k- \rho_k\right]\left[e_{\alpha,1}\left(-\left(\lambda_k+m\right)t^\alpha;q\right)-1\right]}{ \left[1-e_{\alpha,1}\left(-\left(\lambda_k+m\right)T^\alpha;q\right)\right]}.
\end{eqnarray}

Therefore, by using  (\ref{additive4.5}),   (\ref{additive4.6}),  (\ref{additive5.9}) and (\ref{additive5.10})  we get that 
\begin{eqnarray*} 
u(t)&\overset{ \text{(\ref{additive4.5})-(\ref{additive5.10})}}{=}&\varphi 
+\sum\limits_{k\in{I}}\frac{\left[ \varphi_k- \rho_k\right]\left[e_{\alpha,1}\left(-\left(\lambda_k+m\right)t^\alpha;q\right)-1\right]\psi_k}{ \left[1-e_{\alpha,1}\left(-\left(\lambda_k+m\right)T^\alpha;q\right)\right]}, 
\end{eqnarray*}
 and 
\begin{eqnarray*}  
f&\overset{ \text{(\ref{additive4.6})-(\ref{additive5.9})}}{=}& \mathcal{L}\varphi-\sum\limits_{k\in{I}}\frac{\left(\lambda_k+m\right)\left[\varphi_k- \rho_k\right]\psi_k}{1-e_{\alpha,1}\left(-\left(\lambda_k+m\right)T^\alpha;q\right)}.
\end{eqnarray*}

{\it Convergence}. For $0<t<T$, using the estimate   (\ref{additive3.1}), we get
\begin{eqnarray*} 
  1-e_{\alpha,1}\left(-\left(\lambda_k+m\right)T^\alpha;q\right) &\overset{ \text{(\ref{additive3.1})}}{\geq}& 1-\frac{1}{1+\Gamma_q\left(\alpha+1\right)^{-1}\left(\lambda_k+m\right)T^\alpha} \nonumber\\
&=&\frac{\Gamma_q\left(\alpha+1\right)^{-1}\left(\lambda_k+m\right)T^\alpha}{1+\Gamma_q\left(\alpha+1\right)^{-1}\left(\lambda_k+m\right)T^\alpha}.
\end{eqnarray*}

Therefore
\begin{equation}\label{additive5.11}
 0\leq1-e_{\alpha,1}\left(-\left(\lambda_k+m\right)t^\alpha;q\right) <1,\;\;\;0<t\leq T.
\end{equation}

From (\ref{additive5.9}), 
 (\ref{additive5.10}) and (\ref{additive5.11})   we conclude that
\begin{eqnarray}\label{additive5.12}
|u_k|\overset{ \text{(\ref{additive5.10})(\ref{additive5.11})}}{\lesssim} 
 |\varphi_k|+|\rho_k|;\;\;\;  |f_k|\overset{ \text{(\ref{additive5.9})(\ref{additive5.11})}}{\lesssim} 
 \left(\lambda_k+1\right)|\varphi_k|+\left(\lambda_k+1\right)|\rho_k|. 
 \end{eqnarray}

We assume that $\varphi,\rho\in \mathcal{H}^{d+2}_{\mathcal{L}}$.  Then using (\ref{additive4.2}), (\ref{additive4.9}) and (\ref{additive5.12}) we derive that
\begin{eqnarray}\label{additive5.13}  
\|f\|^2_{\mathcal{H}^d_{\mathcal{L}}}&\overset{ \text{(\ref{additive4.2})}}{=}&\sum\limits_{k\in I}\left|\left(1+\lambda_k\right)^{d/2}\langle f, \psi_k\rangle_H\right|^2   \nonumber\\
&\overset{ \text{(\ref{additive5.12})}}{\lesssim}&\sum\limits_{k\in I}\left|\left(1+\lambda_k\right)^{d/2+1}\langle \varphi, \psi_k\rangle_H\right|^2+\sum\limits_{k\in I}\left|\left(1+\lambda_k\right)^{d/2+1}\langle \rho, \psi_k\rangle_H\right|^2\nonumber\\
&\overset{ \text{(\ref{additive4.2})}}{=}&\| \varphi\|^2_{\mathcal{H}^{d+2}_{\mathcal{L}}}+\| \rho\|^2_{\mathcal{H}^{d+2}_{\mathcal{L}}}<\infty,  
\end{eqnarray}
and
\begin{eqnarray}\label{additive5.14} 
\|\left(I+\mathcal{L}\right)u(t)\|^2_{\mathcal{H}^d_{\mathcal{L}}}&\overset{ \text{(\ref{additive4.2})}}{=}&\sum\limits_{k\in I}\left|\left(1+\lambda_k\right)^{d/2+1}\langle  u(t), \psi_k\rangle_H\right|^2    \nonumber\\
&\overset{ \text{(\ref{additive5.12})}}{\lesssim}& \sum\limits_{k\in{I}}\left[|\left(1+\lambda_k\right)^{d/2+1}\langle\varphi, \psi_k\rangle_H|^2+|\left(1+\lambda_k\right)^{d/2+1}\langle\rho, \psi_k\rangle_H|^2\right]\nonumber\\
&\overset{ \text{(\ref{additive4.2})}}{=}&\| \varphi\|^2_{\mathcal{H}^{d+2}_{\mathcal{L}}}+\| \rho\|^{d+2}_{\mathcal{H}^{d+2}_{\mathcal{L}}}<\infty. 
\end{eqnarray}

By using (\ref{additive4.9}), (\ref{additive5.13}) and (\ref{additive5.14}) we find that 
\begin{eqnarray*}
 \|^cD^\alpha_{q,t}u(t)\|^2_{\mathcal{H}^{d}_{\mathcal{L}}}
&\overset{ \text{(\ref{additive4.9})}}{\lesssim}& \|\left(I+\mathcal{L}\right)u(t)\|_{\mathcal{H}^{d}_{\mathcal{L}}}+\|f\|_{\mathcal{H}^d_{\mathcal{L}}}\\
&\overset{ \text{(\ref{additive5.13})(\ref{additive5.14})}}{\lesssim}&\| \varphi\|^2_{\mathcal{H}^{d+2}_{\mathcal{L}}}+\| \rho\|^{d+2}_{\mathcal{H}^{d+2}_{\mathcal{L}}}<\infty,
\end{eqnarray*} 
which means that $ u\in L^\infty_q\left([0, T]; \mathcal{H}^{d+2}_{\mathcal{L}}\right)\cap W^\alpha_q\left([0, T]; \mathcal{H}^d_{\mathcal{L}}\right)$ and $f\in\mathcal{H}^d_{\mathcal{L}}$,  and  the estimates (\ref{additive5.4}) hold.

{\it Uniqueness}. The obtained solution  (\ref{additive5.10})-(\ref{additive5.11}) satisfies the equation (\ref{additive5.1}) and the conditions (\ref{additive5.2})-(\ref{additive5.3}).

We denote  $v(t):=u_1(t)-u_2(t)$ and $f:=f_1-f_2$. Then   $u(t)$ and $f$ satisfy (\ref{additive5.1})  and homogeneous conditions (\ref{additive5.2}) 
and (\ref{additive5.3}).

Applying the inner product on Hilbert space $H$ and the equation (\ref{additive5.1}) we obtain that
\begin{eqnarray*}
\left(1+\lambda_k\right)^cD^\alpha_{q,t} u_k(t)&=&\langle ^cD^\alpha_{q,t}\left[ u(t)+\mathcal{L}u(t)\right], \psi_k\rangle_H\nonumber\\
&=&\langle-mu(t)+f,\psi_k\rangle_H\nonumber\\
&=&-mu_k(t)+f_k. 
\end{eqnarray*}

Taking into account the homogeneous conditions (\ref{additive5.2}) and (\ref{additive5.3}), we find that  $u_k(0) = 0$ and  $u_k(T)=0$.  Therefore $f_k\equiv0$ and $u_k \equiv0$ for $k\in{I}$.

Further, by the completeness of the system $\{\psi_k\}_{k\in{I}}$ in $H$  we obtain
\begin{eqnarray*}
f (t) \equiv0, \;\;\;u (t) \equiv0,  
\end{eqnarray*}
for $t>0$.

Hence, the uniqueness of the solution (\ref{additive5.10})-(\ref{additive5.11}) is proved. The proof is complete.

\end{proof}

\begin{center}

\end{center}

\end{document}